\documentclass[a4paper,11pt,fleqn]{article}
\usepackage{srcltx}
\usepackage{xcolor}

\usepackage{charter}
\usepackage{graphicx}
\usepackage{amsmath}
\usepackage{amssymb}
\usepackage{theorem}
\usepackage{euscript}
\usepackage{epic,eepic}
\usepackage[usenames,dvipsnames]{pstricks}
\usepackage{mathtools} 
\definecolor{labelkey}{rgb}{0,0.08,0.45}
\definecolor{refkey}{rgb}{0,0.6,0.0}
\definecolor{Brown}{rgb}{0.45,0.0,0.05}
\definecolor{dgreen}{rgb}{0.00,0.49,0.00}
\definecolor{dblue}{rgb}{0,0.08,0.75}
\newcommand{\email}[1]{\href{mailto:#1}{\nolinkurl{#1}}}
\newcommand{\minimize}[2]{\ensuremath{\underset{\substack{{#1}}}%
{\text{\rm minimize}}\;\;#2 }}
\newcommand{\argmin}[2]{\ensuremath{\underset{\substack{{#1}}}%
{\text{\rm argmin}}\;\;#2 }}
\newcommand{\Frac}[2]{\displaystyle{\frac{#1}{#2}}} 
\newcommand{\Scal}[2]{{\bigg\langle{{#1}\:\bigg |~{#2}}\bigg\rangle}}
\newcommand{\scal}[2]{{\left\langle{{#1}\mid{#2}}\right\rangle}}

\newcommand{\menge}[2]{\big\{{#1}~\big |~{#2}\big\}} 

\newcommand{\HH}{\ensuremath{{\mathcal H}}}
\newcommand{\BB}{\ensuremath{{\mathcal B}}}
\newcommand{\HHH}{\ensuremath{\boldsymbol{\mathcal H}}}

\newcommand{\GGG}{\ensuremath{\boldsymbol{\mathcal G}}}
\newcommand{\KKK}{\ensuremath{\boldsymbol{\mathcal K}}}

\newcommand{\GG}{\ensuremath{{\mathcal G}}}

\newcommand{\emp}{\ensuremath{{\varnothing}}}

\newcommand{\Id}{\ensuremath{\operatorname{Id}}\,}
\newcommand{\RR}{\ensuremath{\mathbb{R}}}

\newcommand{\bdry}{\ensuremath{\text{bdry}\,}}

\newcommand{\RPP}{\ensuremath{\left]0,+\infty\right[}}

\newcommand{\RX}{\ensuremath{\left]-\infty,+\infty\right]}}
\newcommand{\NN}{\ensuremath{\mathbb N}}

\newcommand{\exi}{\ensuremath{\exists\,}}\,
\newcommand{\pinf}{\ensuremath{{+\infty}}}
\newcommand{\weakly}{\ensuremath{\rightharpoonup}}
\newcommand{\weakc}[1]{\ensuremath{\,{\xrightharpoonup{#1}\,}}}
\newcommand{\stro}[1]{\ensuremath{\,{\xrightarrow{#1}\,}}}

\newcommand{\ran}{\ensuremath{\text{range}}}
\newcommand{\prox}{\ensuremath{\text{\rm prox}}}

\newcommand{\inte}{\ensuremath{\text{int}\,}}
\newcommand{\zeroun}{\ensuremath{\left]0,1\right[}}   
\renewcommand{\leq}{\ensuremath{\leqslant}}
\renewcommand{\geq}{\ensuremath{\geqslant}}

\newtheorem{theorem}{Theorem}[section]

\newtheorem{proposition}[theorem]{Proposition}
\theoremstyle{plain}{\theorembodyfont{\rmfamily}%
}
\theoremstyle{plain}{\theorembodyfont{\rmfamily}%
}
\newtheorem{primal-dual approach}[theorem]{Primal-dual approach}
\theoremstyle{plain}{\theorembodyfont{\rmfamily}%
}
\theoremstyle{plain}{\theorembodyfont{\rmfamily}%
\newtheorem{remark}[theorem]{Remark}}
\theoremstyle{plain}{\theorembodyfont{\rmfamily}%
}
\theoremstyle{plain}{\theorembodyfont{\rmfamily}%
\newtheorem{problem}[theorem]{Problem}}

\numberwithin{equation}{section}

\begin{document}
\title{\sffamily A Strongly Convergent Primal-Dual Method for\\
Nonoverlapping Domain Decomposition\footnote{Contact author: 
P. L. Combettes, {\ttfamily plc@ljll.math.upmc.fr},
phone: +33 1 4427 6319, fax: +33 1 4427 7200.
The work of H. Attouch was supported by ECOS under grant C13E03,  and
by Air Force Office of Scientific Research, USAF, under grant  FA9550-14-1-0056. 
The work of L. M. Brice\~{n}o-Arias and P. L. Combettes was supported
by MathAmSud under grant N13MATH01.
L. M. Brice\~{n}o-Arias was also supported by Conicyt under grants
Fondecyt 3120054 and Anillo ACT1106.}}
\author{H\'edy Attouch,$^1$ Luis M. Brice\~{n}o-Arias,$^2$ and 
Patrick L. Combettes$^3$
\\[2mm]
\small
$\!^1$Universit\'e Montpellier II\\
\small Institut de Math\'ematiques et de Mod\'elisation de 
Montpellier -- UMR 5149\\
\small F-34095, Montpellier, France 
(\email{hedy.attouch@univ-montp2.fr})\\[1mm]
\small $\!^2$Universidad T\'ecnica Federico Santa Mar\'ia\\
\small Departamento de Matem\'atica\\
\small Santiago, Chile (\email{luis.briceno@usm.cl})\\[1mm]
\small $\!^3$Sorbonne Universit\'es -- UPMC Univ. Paris 06\\
\small UMR 7598, Laboratoire Jacques-Louis Lions\\
\small F-75005, Paris, France (\email{plc@ljll.math.upmc.fr})
}
\date{~}
\tolerance 2500
\maketitle
\begin{abstract}
We propose a primal-dual parallel proximal 
splitting method for solving domain decomposition problems 
for partial differential equations. The problem is formulated via 
minimization of energy functions on the subdomains with
coupling constraints which model various
properties of the solution at the interfaces.
The proposed method can handle a wide range of linear and nonlinear
problems, with flexible, possibly nonlinear, transmission 
conditions across the interfaces. Strong convergence in the energy
spaces is established in this general setting, and without any
additional assumption on the energy functions or the geometry of
the problem. Several examples are presented.
\end{abstract}


\newpage

\section{Introduction}

One of the main objectives of domain decomposition is to 
solve partial differential equations and the associated boundary 
value problems on complex geometries by partitioning the 
original domain in smaller and simpler subdomains 
\cite{DDMS13,DDMS09,ChMa,LeTal,PL-Lions,Quar99,Tose05}. 
The objective of the 
present paper is to propose an original algorithm for solving
variational formulations associated with partial 
differential equations posed on partitioned domains. 
Our analysis pertains to 
non-overlapping domain decompositions, in which subdomains intersect
only on their interfaces. The original domain $\Omega$ is
partitioned into $m$ subdomains $(\Omega_i)_{i\in I}$, the
interface between two subdomains $\Omega_i$
and $\Omega_j$ is denoted by $\Upsilon_{\!ij}$, and 
$\Upsilon_{\!ii}$ stands for the part of the boundary of 
$\Omega_i$ shared with the boundary of $\Omega$ 
(see Fig.~\ref{fig:example}, where $I=\{1,\ldots,m\}$). 

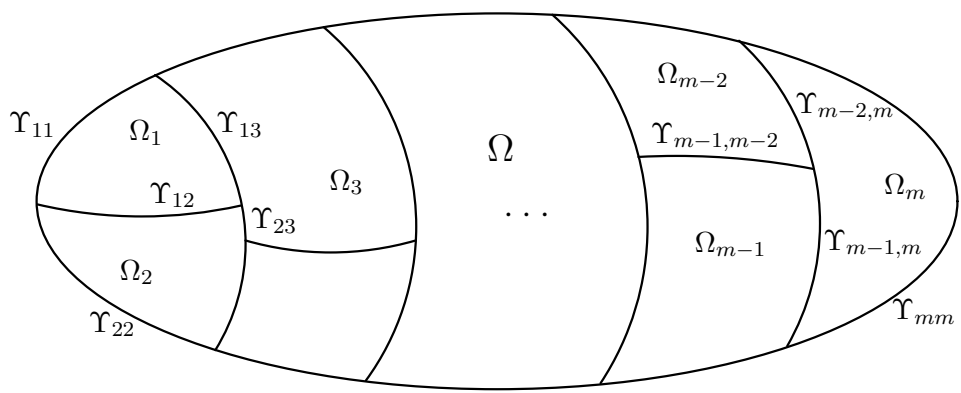
\begin{figure}[!ht]
\begin{center}
\setlength{\unitlength}{0.0176mm}
%
{\renewcommand{\dashlinestretch}{30}
\begin{picture}(6900,3127)(0,-10)
\thicklines
\put(793.063,4700.694){\arc{6785.184}{1.3506}{1.8028}}
\put(36.492,1102.909){\arc{3069.366}{5.3061}{6.8330}}
\put(2207.000,3323.250){\arc{4569.684}{1.2883}{1.8533}}
\put(868.090,1233.208){\arc{3957.029}{5.4206}{6.9104}}
\put(2423.509,1231.296){\arc{4305.230}{5.4461}{6.8624}}
\put(3996.133,1258.798){\arc{3738.473}{5.4549}{6.8013}}
\put(4794.000,-3600.250){\arc{10733.752}{4.6591}{4.9034}}
\put(3450,1425){\ellipse{6884}{2834}}
\put(630,800){\makebox(0,0)[lb]{$\Omega_2$}}
\put(2200,1480){\makebox(0,0)[lb]{$\Omega_3$}}
\put(4650,2270){\makebox(0,0)[lb]{$\Omega_{m-2}$}}
\put(6347,1444){\makebox(0,0)[lb]{$\Omega_m$}}
\put(700,1850){\makebox(0,0)[lb]{$\Omega_{1}$}}
\put(400,410){\makebox(0,0)[lb]{$\Upsilon_{22}$}}
\put(6400,500){\makebox(0,0)[lb]{$\Upsilon_{mm}$}}
\put(3505,1309){\makebox(0,0)[lb]{. . .}}
\put(5680,2029){\makebox(0,0)[lb]{$\Upsilon_{m-2,m}$}}
\put(4600,1800){\makebox(0,0)[lb]{$\Upsilon_{m-1,m-2}$}}
\put(4930,994){\makebox(0,0)[lb]{$\Omega_{m-1}$}}
\put(5890,994){\makebox(0,0)[lb]{$\Upsilon_{m-1,m}$}}
\put(-200,1910){\makebox(0,0)[lb]{$\Upsilon_{11}$}}
\put(850,1380){\makebox(0,0)[lb]{$\Upsilon_{12}$}}
\put(1350,1910){\makebox(0,0)[lb]{$\Upsilon_{13}$}}
\put(1607,1189){\makebox(0,0)[lb]{$\Upsilon_{23}$}}
\put(3375,1729){\makebox(0,0)[lb]{\Large$\Omega$}}
\end{picture}
}
\caption{Decomposition of the domain $\Omega$.}
\label{fig:example}
\end{center}
\end{figure}
A sizable literature has been devoted to variational domain
decomposition; see for instance
\cite{AttBolRedSou,Atto11,Bade06,DDMS13,DDMS09,Chan90,Forn09,%
Glow89,Quar99,Tose05}. The novelty of our framework is
to allow for the use of several subdomains with
general convex energy functions on each of them, together with a
broad range of transmission conditions on interfaces. More
specifically, in our model the $i$th variable $u_i$ lies in a
suitable Sobolev space $\HH_i$ and the structured minimization
problem under consideration assumes the form
\begin{equation}
\label{e:primal}
\minimize{(u_i)_{i\in I}\in\bigoplus_{i\in I}\HH_i}{
\sum_{i\in I}\varphi_i(u_i)+
\sum_{(i,j)\in K}\psi_{ij} 
({\mathsf T}_{\!ij}\,u_i-{\mathsf T}_{\!ji}\,u_j)},
\end{equation}
where $K$ is the set indices of active interfaces, 
$\mathsf{T}_{\!ij}\colon\HH_i\to L^2(\Upsilon_{ij})$ 
denotes the trace operator relative to the interface
$\Upsilon_{ij}$, and $\varphi_i\colon\HH_i\to\RX$ and 
$\psi_{ij}\colon L^2(\Upsilon_{\!ij})\to\RX$ are lower 
semicontinuous convex functions. In applications, one is
often interested in solving the Fenchel-Rockafellar dual 
problem associated with \eqref{e:primal}, the solutions of 
which model tensions (e.g., stresses or fluxes) at the interfaces. 
There are two main components 
in \eqref{e:primal}. The first component is the separable function
$(u_i)_{i\in I}\mapsto\sum_{i\in I}\varphi_i(u_i)$ which
incorporates the internal energy functions 
$(\varphi_i)_{i\in I}$ on each
subdomain. The other component is a coupling term which models
transmission conditions across the interfaces. 
Since the separable term needs not be smooth and may take on 
the value $\pinf$, hard constraints on $(u_i)_{i\in I}$ can be
imposed in our formulation. 
It can also deal with non quadratic functions, 
capturing, for instance, $p$-Laplacian or obstacle problems. 
On the other hand, the coupling function models transmission 
conditions, in particular continuity, through the interfaces. 
A major 
advantage of this approach is its flexibility,
which makes it possible to treat in a unified fashion unilateral 
and/or nonlinear transmission conditions. 

To solve \eqref{e:primal} and its dual, we bring into play a 
multivariate primal-dual proximal splitting method recently 
proposed in \cite{Jnca14} for structured convex minimization 
problems. The algorithm generates both primal and dual sequences 
which converge strongly to the unique solution satisfying the 
Kuhn-Tucker conditions, and lying closest to some initial point.
At each iteration an outer approximation to the Kuhn-Tucker set 
is constructed as the intersection of two half-spaces, and the 
update is obtained by projecting the initial point onto this 
intersection. 
This method will be adapted to solve the 
variational problem \eqref{e:primal} in a fully 
split fashion, in that each elementary step of the algorithm 
involves the constituents of the problem (namely $u_i$, $\varphi_i$,
$\psi_{ij}$, and $\mathsf{T}_{ij}$) separately. In addition, its
structure lends it to implementations on parallel architectures.
Let us note that typically, Lagrangian-based approaches 
\cite{Atto09,Glow89} do not achieve full splitting with respect 
to the linear operators, which complicates the numerical 
implementation and may require additional restrictions on these 
linear operators to ensure convergence. Another salient advantage 
of the proposed algorithm that distinguishes it from 
Lagrangian-based approaches as well as from splitting algorithms 
which could be considered for solving \eqref{e:primal}, such as 
those of \cite{Bot13c,Siop11,Siop13,Svva12,Cond13,Bang13}, is that 
these methods provide only weak convergence. In addition,
the methods of \cite{Bot13c,Siop11,Siop13,Svva12,Cond13,Bang13}
require the computation of bounds on the range of certain
parameters. In the case of \eqref{e:primal}, these bounds involve 
norms of combinations of trace operators, which are very hard to 
estimate. Altogether, the proposed algorithm provides significant
advantages over the state of the art.

The paper is organized as follows. In Section~\ref{sec:2},
we present the notation and the abstract primal-dual splitting 
algorithm which is the basis of our method. 
In Section~\ref{sec:3}, we formally state the domain decomposition 
problem under investigation, define the functional setting,
and introduce the main algorithm. Section~\ref{sec:4} is devoted 
to applications to concrete domain decomposition problems. 
Finally, in Section~\ref{sec:5}, we briefly discuss some adaptations 
of our setting to other interesting problems.

\section{Notation and preliminaries}
\label{sec:2}

Let $\BB$ be a real Banach space.
Weak and strong convergence in $\BB$ are denoted by 
$\weakc{\BB}$ and $\stro{\BB}$, respectively, and 
$\Gamma_0(\BB)$ is the class of lower semicontinuous 
convex functions $\varphi\colon\BB\to\RX$ which are not 
identically equal to $\pinf$. A function $\varphi\colon\BB\to\RX$
is coercive if $\lim_{\|u\|\to\pinf}\varphi(u)=\pinf$. The Hilbert
direct sum of a finite family of Hilbert spaces $(\HH_i)_{i\in I}$
is denoted by $\bigoplus_{i\in I}\HH_i$.

$\RR^N$ denotes the usual $N$-dimensional Euclidean space and 
$|\cdot|$ its norm. Let $\Omega$ be a nonempty open 
bounded subset of $\RR^N$ with Lipschitz boundary $\bdry\Omega$.
We denote by $x$ a generic element of $\Omega$, and by $dx$ the
restriction to $\Omega$ of the Lebesgue measure on $\RR^N$.
All the functional spaces considered throughout the paper involve
real-valued functions. For every $p\in\left]1,\pinf\right[$, 
$W^{1,p}(\Omega)=\menge{v\in L^p(\Omega)}{Dv\in(L^p(\Omega))^N}$, 
where $D$ denotes the weak gradient (derivatives in the sense of
distributions). In particular, we set 
$H^1(\Omega)=W^{1,2}(\Omega)$, which is a Hilbert space 
with scalar product $\scal{\cdot}{\cdot}_{H^1(\Omega)}\colon(u,v)
\mapsto\int_{\Omega}uv+\int_{\Omega}(Du)^\top Dv$. 
We denote by $S$ the
surface measure on $\bdry\Omega$ \cite[Section~1.1.3]{Neca67}. 
Now let $\Upsilon$ be a nonempty open subset of
$\bdry\Omega$ 
and let $L^2(\Upsilon)$ be the space of square $S$-integrable 
functions on $\Upsilon$. Endowed with the scalar product 
$(v,w)\mapsto\int_{\Upsilon}vw\,dS$, 
$L^2(\Upsilon)$ is a Hilbert space. The Sobolev trace operator 
${\mathsf T}\colon H^1(\Omega)\to L^2(\bdry\Omega)$ is the
 unique bounded linear operator such that 
$(\forall v\in{\EuScript C}^1(\overline{\Omega}))$ 
${\mathsf T}v=v|_{\bdry\Omega}$. 
Endowed with the scalar product 
\begin{equation}
\label{e:scaledp}
\scal{\cdot}{\cdot}\colon(u,v)\mapsto\int_{\Omega}(D u)^\top D v,
\end{equation}
the space $H^1_{0,\Upsilon}(\Omega)=\menge{u\in H^1(\Omega)}
{{\mathsf T}\,u=0\:\;\text{on}\:\Upsilon}$
is a Hilbert space \cite[Section~25.10]{Zeid90B}. For every 
$\alpha\in\left]0,1\right]$, 
${\EuScript C}^{1,\alpha}(\overline{\Omega})$ is the
subspace of ${\EuScript C}^1(\overline{\Omega})$ consisting of 
those functions $u$ such that 
\begin{equation}
(\exi\mu\in\RPP)(\forall (x,y)\in\Omega^2)
\quad|u(x)-u(y)|\leq\mu|x-y|^{\alpha}
\quad\text{and}\quad |Du(x)-Du(y)|\leq\mu|x-y|^{\alpha}.
\end{equation}
Finally,
for $S$-almost every $\omega\in\bdry\Omega$, there exists a unit 
outward normal vector $\nu(\omega)$. 
For details and complements, see 
\cite{Adam03,Atto06,Drab07,Gris85,Neca67,Zeid90A,Zeid90B}.

Let $\HH$ be a real Hilbert space with scalar product
$\scal{\cdot}{\cdot}$ and associated norm $\|\cdot\|$,
and let $\varphi\in\Gamma_0(\HH)$.
The subdifferential of $\varphi$ is 
\begin{equation}
\label{e:subdiff}
\partial\varphi\colon\HH\to 2^{\HH}\colon 
u\mapsto\menge{u^*\in\HH}{(\forall v\in\HH)\:\:
\varphi(u)+\scal{v-u}{u^*}\leq\varphi(v)},
\end{equation}
the conjugate of $\varphi$ is the function 
$\varphi^*\in\Gamma_0(\HH)$ defined by
\begin{equation}
\varphi^*\colon
u^*\mapsto\sup_{u\in\HH}\big(\scal{u}{u^*}-\varphi(u)\big),
\end{equation}
and the proximity operator of $\varphi\in\Gamma_0(\HH)$ is 
\cite{Mor62b}
\begin{equation}
\label{e:prox}
\prox_{\varphi}\colon\HH\to\HH\colon u\mapsto\argmin{v\in\HH}
{\bigg(\varphi(v)+\frac12\|u-v\|^2\bigg)}.
\end{equation}
Let $C$ be a nonempty closed convex subset of $\HH$. The indicator 
function of $C$ is
\begin{equation}
\iota_C\colon\HH\to\RX\colon u\mapsto
\begin{cases}
0,&\text{if}\;\;u\in C;\\
\pinf,&\text{otherwise,}
\end{cases}
\end{equation}
and the projection (or best approximation) operator onto $C$ is
\begin{equation}
P_C=\prox_{\iota_C}\colon\HH\to C\colon u\mapsto
\argmin{v\in C}{\|u-v\|}.
\end{equation}
For background on convex analysis in Hilbert spaces the reader is 
referred to \cite{Livre1}.

\newpage
The backbone of our model will be the following abstract 
primal-dual saddle problem.

\begin{problem}
\label{prob:5}
Let $I$ and $K$ be nonempty finite index sets, and let
$(\HH_i)_{i\in I}$ and $(\GG_k)_{k\in K}$ be real 
Hilbert spaces. For every $i\in I$ and $k\in K$, let 
$\Phi_i\in\Gamma_0(\HH_i)$, let $\Psi_k\in\Gamma_0(\GG_k)$, let 
$\Lambda_{ki}\colon\HH_i\to\GG_k$ be a bounded linear operator, and
let $\Lambda^*_{ki}\colon\GG_k\to\HH_i$ be its adjoint. 
It is assumed that
\begin{equation}
\label{e:2012-10-21a}
(\forall i\in I)\quad
0\in\ran\bigg(\partial\Phi_i+\sum_{k\in K}\Lambda_{ki}^*\circ
(\partial\Psi_k)\circ\sum_{j\in I}\Lambda_{kj}\bigg).
\end{equation}
Let $\boldsymbol{u}_0=(u_{i,0})_{i\in I}\in\HHH=
\bigoplus_{i\in I}\HH_i$ and let $\boldsymbol{w}_0=
(w_{k,0})_{k\in K}\in\GGG=\bigoplus_{k\in K}\GG_k$.
The problem is to find the best approximation in $\HHH\oplus\GGG$
to $(\boldsymbol{u}_0,\boldsymbol{w}_0)$ from the Kuhn-Tucker set
\begin{multline}
\label{e:2013-12-07k}
\boldsymbol{Z}=\Bigg\{\boldsymbol{u}=(u_i)_{i\in I}\in\HHH
,\;\boldsymbol{w}=(w_k)_{k\in K}\in\GGG\;\bigg |\;
(\forall i\in I)\;\;-\sum_{k\in K}\Lambda_{ki}^*w_k\in
\partial\Phi_i(u_i)\;\;\\
\text{and} \;\;  (\forall k\in K)\;\;\sum_{i\in I}\Lambda_{ki}u_i\in
\partial\Psi_k^*(w_k)\Bigg\}.
\end{multline}
\end{problem}

\begin{proposition}
\label{p:2014-07-02}
Problem~\ref{prob:5} has a unique solution
$(\overline{\boldsymbol{u}},\overline{\boldsymbol{w}})$.
Moreover, $\overline{\boldsymbol{u}}=(\overline{u}_i)_{i\in I}$ 
solves the primal problem
\begin{equation}
\label{e:2012-10-23p}
\minimize{(u_i)_{i\in I}\in\:
\bigoplus_{i\in I}\HH_i}{\sum_{i\in I}
\Phi_i(u_i)+\sum_{k\in K}\Psi_k\bigg(\sum_{i\in I}
\Lambda_{ki}u_{i}\bigg)},
\end{equation}
and $\overline{\boldsymbol{w}}=(\overline{w}_{k})_{k\in K}$ solves
the dual problem
\begin{equation}
\label{e:2012-10-23d}
\minimize{(w_k)_{k\in K}\in\:\bigoplus_{k\in K}\GG_k}{\sum_{i\in I}
\Phi_i^*\bigg(-\sum_{k\in K}
\Lambda_{ki}^*w_k\bigg)+\sum_{k\in K}\Psi^*_k(w_k)}.
\end{equation}
\end{proposition}
\begin{proof}
Since $\boldsymbol{Z}$ in \eqref{e:2013-12-07k} is nonempty, 
closed, and convex \cite[Proposition~2.8]{Siop11}, the projection 
$(\overline{\boldsymbol{u}},\overline{\boldsymbol{w}})$ 
of $(\boldsymbol{u}_0,\boldsymbol{w}_0)$ onto $\boldsymbol{Z}$ 
is uniquely defined. The remaining claims follow from 
\cite[Corollary~4.5(i)]{Jnca14}.
\end{proof}

\medskip

To solve Problem~\ref{prob:5}, we shall use the following splitting
algorithm from \cite{Jnca14}. This algorithm generates a
sequence $(\boldsymbol{u}_{n},\boldsymbol{w}_{n})_{n\in\NN}$ that 
converges strongly to the unique solution to Problem~\ref{prob:5}.
It exploits a convergence principle that goes back in its simplest
form to the work of Haugazeau \cite{Haug68} (see \cite{Sico00} for 
historical comments). Let us note that existing methods for solving
\eqref{e:2012-10-23p}--\eqref{e:2012-10-23d} 
\cite{Bot13c,Siop11,Siop13,Svva12,Cond13,Bang13} guarantee only weak 
convergence to an unspecified primal-dual solution and, in 
addition, require
the knowledge of bounds on certain compositions of the linear
operators involved in the model. In our setting, such bounds would
be extremely hard to obtain. Moreover, the proposed method
solves Problem~\ref{prob:5} in a fully split fashion in that each 
elementary step of the algorithm activates the functions and 
operators of the problem separately. 

In geometrical terms, the algorithm is executed as follows 
\cite{Jnca14}. Set 
$\boldsymbol{x}_0=(\boldsymbol{u}_0,\boldsymbol{w}_0)$ and,
given two points $\boldsymbol{a}$ and 
$\boldsymbol{b}$ in $\KKK=\HHH\oplus\GGG$, denote by 
$H(\boldsymbol{a},\boldsymbol{b})$ the closed affine half-space of 
$\KKK$ onto which $\boldsymbol{b}$ is the projection of 
$\boldsymbol{a}$. At iteration $n$, the current iterate is
$\boldsymbol{x}_n=\big((u_{i,n})_{i\in I},(w_{k,n})_{k\in K}\big)
\in\KKK$ and we find $\boldsymbol{x}_{n+1/2}=
\big((u_{i,n+1/2})_{i\in I},(w_{k,n+1/2})_{k\in K}\big)\in\KKK$ 
such that 
$\boldsymbol{Z}\subset H(\boldsymbol{x}_n,\boldsymbol{x}_{n+1/2})$.
The update $\boldsymbol{x}_{n+1}=
\big((u_{i,n+1})_{i\in I},(w_{k,n+1})_{k\in K}\big)$ is then
obtained as the projection of $\boldsymbol{x}_0$
onto $H(\boldsymbol{x}_0,\boldsymbol{x}_{n})\cap
H(\boldsymbol{x}_n,\boldsymbol{x}_{n+1/2})$, which can be computed
explicitly in terms of 
$(\boldsymbol{x}_0,\boldsymbol{x}_{n},\boldsymbol{x}_{n+1/2})$.
The computation of $\boldsymbol{x}_{n+1/2}$ involves proximal steps
with respect to the functions $(\Phi_i)_{i\in I}$ and 
$(\Psi_k)_{k\in K}$, as well as applications of the linear operators
$(\Lambda_{ki})_{i\in I, k\in K}$ and their adjoints.

\newpage
\begin{theorem}
\label{t:5}
{\rm\cite[Corollary~4.5(ii)--(iii)]{Jnca14}}
Let $\varepsilon\in\zeroun$, let $(\gamma_n)_{n\in\NN}$ and 
$(\mu_n)_{n\in\NN}$ be sequences in $[\varepsilon,1/\varepsilon]$, 
let $(\lambda_n)_{n\in\NN}$ be a sequence in $[\varepsilon,1]$, 
and iterate
\begin{equation}
\label{e:2014-06-05a}
\begin{array}{l}
\text{for}\;n=0,1,\ldots\\
\left\lfloor
\begin{array}{l}
\text{for every}\;i\in I\\
\left\lfloor
\begin{array}{l}
v_{i,n}=u_{i,n}-\gamma_n\sum_{k\in K}\Lambda_{ki}^*w_{k,n}\\
p_{i,n}=\prox_{\gamma_n\Phi_i}v_{i,n}\\
\end{array}
\right.\\
\text{for every}\;k\in K\\
\left\lfloor
\begin{array}{l}
l_{k,n}=\sum_{i\in I}\Lambda_{ki}u_{i,n}\\
q_{k,n}=\prox_{\mu_n\Psi_k}\big(l_{k,n}+\mu_nw_{k,n}\big)\\
t_{k,n}=q_{k,n}-\sum_{i\in I}\Lambda_{ki}p_{i,n}\\
\end{array}
\right.\\
\text{for every}\;i\in I\\
\left\lfloor
\begin{array}{l}
s_{i,n}=\gamma_n^{-1}(u_{i,n}-p_{i,n})+
\mu_n^{-1}\sum_{k\in K}\Lambda_{ki}^*(l_{k,n}-q_{k,n})\\
\end{array}
\right.\\
\tau_n=\sum_{i\in I}\|s_{i,n}\|^2+\sum_{k\in K}\|t_{k,n}\|^2\\
\text{if}\;\tau_n=0\\
\left\lfloor
\begin{array}{l}
\theta_n=0\\
\end{array}
\right.\\
\text{if}\;\tau_n>0\\
\left\lfloor
\begin{array}{l}
\theta_n=\lambda_n\big(\gamma_n^{-1}\sum_{i\in I}
\|u_{i,n}-p_{i,n}\|^2+\mu_n^{-1}
\sum_{k\in K}\|l_{k,n}-q_{k,n}\|^2\big)/\tau_n\\
\end{array}
\right.\\
\text{for every}\;i\in I\\
\left\lfloor
\begin{array}{l}
u_{i,n+1/2}=u_{i,n}-\theta_n s_{i,n}\\
\end{array}
\right.\\
\text{for every}\;k\in K\\
\left\lfloor
\begin{array}{l}
w_{k,n+1/2}=w_{k,n}-\theta_n t_{k,n}
\end{array}
\right.\\
\chi_n=\sum_{i\in I}\scal{u_{i,0}-u_{i,n}}{u_{i,n}-u_{i,n+1/2}}
+\sum_{k\in K}\scal{w_{k,0}-w_{k,n}}{w_{k,n}-w_{k,n+1/2}}\\
\mu_n=\sum_{i\in I}\|u_{i,0}-u_{i,n}\|^2+\sum_{k\in K}
\|w_{k,0}-w_{k,n}\|^2\\
\nu_n=\sum_{i\in I}\|u_{i,n}-u_{i,n+1/2}\|^2+
\sum_{k\in K}\|w_{k,n}-w_{k,n+1/2}\|^2\\
\rho_n=\mu_n\nu_n-\chi_n^2\\
\text{if}\;\rho_n=0\;\text{and}\;\chi_n\geq 0\\
\left\lfloor
\begin{array}{l}
\text{for every}\;i\in I\\
\left\lfloor
\begin{array}{l}
u_{i,n+1}=u_{i,n+1/2}\\
\end{array}
\right.\\
\text{for every}\;k\in K\\
\left\lfloor
\begin{array}{l}
w_{k,n+1}=w_{k,n+1/2}\\
\end{array}
\right.\\
\end{array}
\right.\\
\text{if}\;\rho_n>0\;\text{and}\;\chi_n\nu_n\geq\rho_n\\
\left\lfloor
\begin{array}{l}
\text{for every}\;i\in I\\
\left\lfloor
\begin{array}{l}
u_{i,n+1}=u_{i,0}+(1+\chi_n/\nu_n)(u_{i,n+1/2}-u_{i,n})\\
\end{array}
\right.\\
\text{for every}\;k\in K\\
\left\lfloor
\begin{array}{l}
w_{k,n+1}=w_{k,0}+(1+\chi_n/\nu_n)(w_{k,n+1/2}-w_{k,n})
\end{array}
\right.\\
\end{array}
\right.\\
\text{if}\;\rho_n>0\;\text{and}\;\chi_n\nu_n<\rho_n\\
\left\lfloor
\begin{array}{l}
\text{for every}\;i\in I\\
\left\lfloor
\begin{array}{l}
u_{i,n+1}=u_{i,n}+(\nu_n/\rho_n)\big(\chi_n(u_{i,0}-u_{i,n})
+\mu_n(u_{i,n+1/2}-u_{i,n})\big)\\
\end{array}
\right.\\
\text{for every}\;k\in K\\
\left\lfloor
\begin{array}{l}
w_{k,n+1}=w_{k,n}+(\nu_n/\rho_n)
\big(\chi_n(w_{k,0}-w_{k,n})
+\mu_n(w_{k,n+1/2}-w_{k,n})\big).
\end{array}
\right.\\
\end{array}
\right.\\
\end{array}
\right.\\
\end{array}
\end{equation}
Then, for every $i\in I$ and every $k\in K$, \eqref{e:2014-06-05a} 
generates infinite sequences $(u_{i,n})_{n\in\NN}$ and 
$(w_{k,n})_{n\in\NN}$ such that $u_{i,n}\stro{\HH_i}\overline{u}_i$
and $w_{k,n}\stro{\GG_k}\overline{w}_k$.
\end{theorem}

\section{Problem formulation and algorithm}
\label{sec:3}

The problem under consideration is the following.

\begin{problem}
\label{prob:1}
Let $\Omega$ be a nonempty open bounded subset of $\RR^N$ with 
Lipschitz boundary $\bdry\Omega$, let $m\geq 2$ be an integer, and 
set $I=\{1,\ldots,m\}$. Suppose that the following hold:
\begin{enumerate}
\item
\label{A2}
$(\Omega_i)_{i\in I}$ are disjoint 
open subsets of $\Omega$ (see Fig.~\ref{fig:example}) with
Lipschitz boundaries $(\bdry\Omega_i)_{i\in I}$, 
$\overline{\Omega}=\bigcup_{i\in I}\overline{\Omega_i}$, and
\begin{equation}
\label{e:15octobre2008}
(\forall i\in I)\quad\Upsilon_{ii}=
\inte_{\bdry\Omega}(\bdry\Omega_i\cap\bdry\Omega)\neq\emp,
\end{equation}
where $\inte_{\bdry\Omega}$ denotes the interior relative to
${\bdry\Omega}$.
\item
\label{A3}
For every $i\in I$, 
\begin{equation}
J(i)=\menge{j\in I\smallsetminus\{i\}}{\Upsilon_{ij}\neq\emp}
\neq\emp,
\end{equation}
where
\begin{equation}
\label{e:24fevrier2009}
(\forall j\in\{i+1,\ldots,m\})\quad\Upsilon_{ij}=\Upsilon_{ji}
=\inte_{\bdry\Omega_i}(\bdry\Omega_i\cap\bdry\Omega_j).
\end{equation}
Moreover, $J(i-)=J(i)\cap\{1,\ldots,i-1\}$ 
and $J(i+)=J(i)\cap\{i+1,\ldots,m\}$, with the convention 
$J(1-)=J(m+)=\emp$.
\item\label{A4}
The set of indices of interfaces is
\begin{equation}
\label{e:nypl2013-06-20a}
K=\menge{(i,j)}{i\in\{1,\ldots,m-1\}\;\;\text{and}\;\;j\in J(i+)}.
\end{equation}
\item
\label{A5}
For every $i\in I$, 
${\mathsf T}_{\!i}\colon H^1(\Omega_i)\to L^2(\bdry\Omega_{i})$ 
is the trace operator. Moreover, 
\begin{equation}
\label{e:HH_iEDP}
\HH_i=H^1_{0,\Upsilon_{ii}}(\Omega_i)=\menge{u\in H^1(\Omega_i)}
{{\mathsf T}_{\!i}\,u=0\:\;\text{on}\:\Upsilon_{ii}},
\end{equation}
endowed with the scalar product
\begin{equation}
\label{scalar}
\scal{u}{v}
=\int_{\Omega_i}(Du)^\top Dv,
\end{equation} 
is a Hilbert space, and, for every $j\in J(i)$, 
${\mathsf T}_{\!ij}\colon\HH_i\to L^2(\Upsilon_{ij})\colon u\mapsto
({\mathsf T}_{\!i}u)|_{\Upsilon_{ij}}$.
\item
\label{A6}
For every $i\in I$,
\begin{equation}
\GG_i=\bigoplus_{j\in J(i)}L^2(\Upsilon_{ij}),
\end{equation}
$\nu_i(\omega)$ is the unit outward normal vector at 
$\omega\in\bdry\Omega_i$, and 
\begin{equation}
\label{e:2011-02-24a}
Q_i\colon L^2(\Omega_i)\times\GG_i\to\HH_i
\end{equation}
is the operator that maps every 
$(f,(h_j)_{j\in J(i)})$ in $L^2(\Omega_i)\times\GG_i$
into the weak solution in $\HH_i$ of the Dirichlet-Neumann 
boundary value problem
\begin{equation}
\label{e:Q_i}
\begin{cases}
-\Delta u=f&\text{on}\;\Omega_i,\\
u=0&\text{on}\;\Upsilon_{ii},\\
\nu_i^\top D u=h_{j}
&\text{on}\;\Upsilon_{ij},\:\:
\text{for every}\:\:j\in J(i+),\\
\nu_i^\top D u=-h_{j}&\text{on}\;\Upsilon_{ij},\:\:
\text{for every}\:\:j\in J(i-).
\end{cases}
\end{equation} 
\item
\label{A7}
For every $(i,j)\in K$, $\varphi_i\in\Gamma_0(\HH_i)$ and
$\psi_{ij}\in\Gamma_0(L^2(\Upsilon_{ij}))$. 
\item
\label{A8}
There exist 
$\widetilde{\boldsymbol{u}}=
(\widetilde{u}_{i})_{i\in I}\in\bigoplus_{i\in I}\HH_i$ and 
$\widetilde{\boldsymbol{g}}=(\widetilde{g}_{ij})_{(i,j)\in K}
\in\bigoplus_{(i,j)\in K}L^2(\Upsilon_{ij})$ such that
\begin{equation}
\label{e:exicond}
(\forall (i,j)\in K)\quad
\begin{cases}
\widetilde{g}_{ij}\in\partial\psi_{ij}
({\mathsf T}_{\!ij}\,\widetilde{u}_i-
{\mathsf T}_{\!ji}\,\widetilde{u}_j)\\
-Q_i\big(0,(\widetilde{g}_{ij})_{j\in J(i+)},
(\widetilde{g}_{ji})_{j\in J(i-)}\big)\in
\partial\varphi_i(\widetilde{u}_i).
\end{cases}
\end{equation}
\item
\label{A9}
Let $\boldsymbol{u}_0=(u_{i,0})_{i\in I}\in\bigoplus_{i\in I}\HH_i$
and let 
$\boldsymbol{g}_0=(g_{ij,0})_{(i,j)\in K}\in\bigoplus_{(i,j)\in K}
L^2(\Upsilon_{ij})$.
\end{enumerate}
The problem is to find the closest point
$(\overline{\boldsymbol{u}},\overline{\boldsymbol{g}})$ to
$(\boldsymbol{u}_0,\boldsymbol{g}_0)$ in 
$\bigoplus_{i\in I}\HH_i\oplus\bigoplus_{(i,j)\in K}
L^2(\Upsilon_{ij})$ that satisfies \eqref{e:exicond}.
\end{problem}

\begin{remark}
In Problem~\ref{prob:1}, \ref{A2}--\ref{A4}
describe the geometrical setting, and \ref{A5}--\ref{A9} fix the
functional Hilbert setting. In particular, item \ref{A8} will
ensure the existence of a solution. For every $i\in I$, since 
$\bdry\Omega_i=\overline{\Upsilon_{ii}}\cup
\overline{\bigcup_{j\in J(i)}\Upsilon_{ij}}$,
the existence and uniqueness of the solution to \eqref{e:Q_i} is
guaranteed by condition~\ref{A2} in Problem~\ref{prob:1} and
\cite[Theorem~25.I]{Zeid90B}, from which we deduce that $Q_i$ is
linear and continuous. 
\end{remark}

In order to analyze and solve Problem~\ref{prob:1}, we shall exploit 
the following connection.

\begin{proposition}
\label{p:0}
Problem~\ref{prob:1} is a special case of Problem~\ref{prob:5}.
\end{proposition}
\begin{proof}
Let us set
\begin{equation}
\label{e:nypl2013-06-25a}
\big(\forall k=(i,j)\in K\big)\quad
\Psi_{k}=\psi_{ij}\quad\text{and}\quad
(\forall\ell\in I)\quad
\Lambda_{k\ell}=
\begin{cases}
\mathsf{T}_{\!ij},&\text{if}\;\;\ell=i;\\
-\mathsf{T}_{\!ji},&\text{if}\;\;\ell=j;\\
0,&\text{otherwise}.
\end{cases}
\end{equation}
We also define 
\begin{equation}
\label{e:erice2013-06-16x}
(\forall i\in I)\quad\Phi_i=\varphi_i.
\end{equation}
For every $i\in I$, it follows from Poincar\'e's inequality, 
that the embedding $\HH_i\hookrightarrow H^1(\Omega_i)$ is 
continuous \cite[p.~1033]{Zeid90B} and therefore, for every 
$j\in J(i)$, the trace operators 
${\mathsf T}_{\!i}\colon H^1(\Omega_i)\to L^2(\bdry\Omega_{i})$ 
and 
${\mathsf T}_{\!ij}\colon\HH_i\to L^2(\Upsilon_{ij})$ are 
linear and bounded. Moreover, for every $i\in I$, every 
$(u_i)_{i\in I}\in\bigoplus_{i\in I}\HH_i$, and every 
\begin{equation}
\label{e:2014-07-02z}
(w_k)_{k\in K}=(g_{ij})_{(i,j)\in K}
\in\bigoplus_{(i,j)\in K}L^2(\Upsilon_{ij}), 
\end{equation}
it follows from \ref{A6} in Problem~\ref{prob:1} that 
\begin{align}
\Scal{u_i}{\sum_{k\in K}\Lambda^*_{k,i}\,w_k}
&=\Scal{u_i}{\sum_{j\in J(i+)}{\mathsf T}^*_{\!ij}g_{ij}
-\sum_{j\in J(i-)}{\mathsf T}^*_{\!ji}g_{ji}}
\nonumber\\
&=\sum_{j\in J(i+)}\scal{{\mathsf T}_{\!ij}u_i}{g_{ij}}
-\sum_{j\in J(i-)}\scal{{\mathsf T}_{\!ji}u_i}{g_{ji}}
\nonumber\\
&=\sum_{j\in J(i+)}\int_{\Upsilon_{ij}}
({\mathsf T}_{\!ij}u_i)g_{ij}\,dS
-\sum_{j\in J(i-)}\int_{\Upsilon_{ij}}
({\mathsf T}_{\!ij}u_i)g_{ji}\,dS\nonumber\\
&=\int_{\bdry \Omega_i}({\mathsf T}_{i}u_i)
\big(\nu_i^{\top}DQ_i(0,(g_{ij})_{j\in J(i+)},(g_{ji})_{j\in
J(i-)})\big)\,dS\nonumber\\
&=\int_{\Omega_i}(Du_i)^\top DQ_i(0,(g_{ij})_{j\in
J(i+)},(g_{ji})_{j\in J(i-)})\nonumber\\
&=\scal{u_i}{Q_i\big(0,(g_{ij})_{j\in J(i+)},
(g_{ji})_{j\in J(i-)}\big)},
\end{align}
which yields
\begin{equation}
\label{e:santiago2012-01-13}
(\forall i\in I)\quad Q_i\big(0,(g_{ij})_{j\in J(i+)},
(g_{ji})_{j\in J(i-)}\big)=
\sum_{k\in K}\Lambda^*_{ki}\,w_{k}.
\end{equation}
It remains to check that \eqref{e:2012-10-21a} is satisfied.  
It follows from \ref{A8} that there exist 
$(\widetilde{u}_i)_{i\in I}\in\bigoplus_{i\in I}\HH_i$ and
$(\widetilde{w}_k)_{k\in K}=(\widetilde{g}_{ij})_{(i,j)\in K}\in
\bigoplus_{(i,j)\in K}L^2(\Upsilon_{ij})$
such that \eqref{e:exicond} holds.  
Combining \ref{A8}, \eqref{e:nypl2013-06-25a}, 
\eqref{e:erice2013-06-16x}, 
and \eqref{e:santiago2012-01-13} we obtain
\begin{eqnarray}
\label{e:santiago2014-01-06a}
\eqref{e:exicond}
&\Leftrightarrow&
\begin{cases}
(\forall i\in I)\quad 
-\sum_{k\in K}\Lambda^*_{ki}\widetilde{w}_{k}
\in\partial\Phi(\widetilde{u}_i)\\
(\forall k\in K)\quad\widetilde{w}_k\in\partial\Psi_k
\big(\sum_{\ell\in I}\Lambda_{k\ell}\widetilde{u}_\ell\big)
\end{cases}\nonumber\\
&\Rightarrow&
0\in\partial\Phi(\widetilde{u}_i)+
\sum_{k\in K}\Lambda^*_{ki}\bigg(\partial\Psi_k
\bigg(\sum_{\ell\in I}\Lambda_{k\ell}\widetilde{u}_\ell\bigg)\bigg)
\nonumber\\
&\Rightarrow&\eqref{e:2012-10-21a},
\end{eqnarray}
which completes  the proof.
\end{proof}

The following proposition clarifies the interplay between
Problem~\ref{prob:1}, \eqref{e:primal}, and its dual.

\begin{proposition}
\label{p:1}
Problem~\ref{prob:1} has a unique solution
$(\overline{\boldsymbol{u}},\overline{\boldsymbol{g}})$. Moreover, 
$\overline{\boldsymbol{u}}=(\overline{u}_i)_{i\in I}$ solves
\begin{equation}
\label{e:erice2013-06-16p}
\minimize{(u_i)_{i\in I}\in\:\underset{i\in I}{\bigoplus}\HH_i}
{\sum_{i\in I}\varphi_i(u_i)+
\sum_{(i,j)\in K}\psi_{ij} 
({\mathsf T}_{\!ij}\,u_i-{\mathsf T}_{\!ji}\,u_j)}
\end{equation}
and $\overline{\boldsymbol{g}}=(\overline{g}_{ij})_{(i,j)\in K}$ 
solves 
\begin{equation}
\label{e:erice2013-06-16d}
\minimize{(g_{ij})_{(i,j)\in K}\in\:
\underset{(i,j)\in K}{\bigoplus}L^2(\Upsilon_{ij})}
{\sum_{i\in I}\varphi_{i}^*
\Big(-Q_i\big(0,(g_{ij})_{j\in J(i+)},(g_{ji})_{j\in J(i-)}
\big)\Big)+\sum_{(i,j)\in K}\psi_{ij}^*(g_{ij})}.
\end{equation}
\end{proposition}
\begin{proof}
This follows from Proposition~\ref{p:0} and 
Proposition~\ref{p:2014-07-02} applied with 
\eqref{e:nypl2013-06-25a}, \eqref{e:erice2013-06-16x}, 
\eqref{e:2014-07-02z}, and \eqref{e:santiago2012-01-13}.
\end{proof}

Our objective is to provide a flexible method for solving
Problem~\ref{prob:1} (and hence \eqref{e:erice2013-06-16p} and 
\eqref{e:erice2013-06-16d}) in which each elementary step 
involves the constituents of the problem, i.e., the trace 
operators and the functions, separately. 
\newpage

\begin{theorem}
\label{t:1}
Let $\varepsilon\in\zeroun$, let $(\gamma_n)_{n\in\NN}$ and 
$(\mu_n)_{n\in\NN}$ be sequences in 
$[\varepsilon,1/\varepsilon]$, let $(\lambda_n)_{n\in\NN}$ be a
sequence in $[\varepsilon,1]$, and iterate
\begin{equation}
\label{e:erice2013-06-16a}
\begin{array}{l}
\text{for}\;n=0,1,\ldots\\
\left\lfloor
\begin{array}{l}
\text{for every}\;i\in I\\
\left\lfloor
\begin{array}{l}
v_{i,n}=u_{i,n}-\gamma_nQ_i
\big(0,(g_{ij,n})_{j\in J(i+)},(g_{ji,n})_{j\in J(i-)}\big)
\\[1mm]
p_{i,n}=\prox_{\gamma_n\varphi_i}v_{i,n}
\end{array}
\right.\\
\text{for every}\;i\in I\\
\left\lfloor
\begin{array}{l}
\text{for every}\;j\in J(i+)\\
\left\lfloor
\begin{array}{l}
l_{ij,n}={\mathsf T}_{\!ij}u_{i,n}-{\mathsf T}_{\!ji}u_{j,n}\\
q_{ij,n}=\prox_{\mu_n\psi_{ij}}(l_{ij,n}+\mu_ng_{ij,n})\\
t_{ij,n}=q_{ij,n}-{\mathsf T}_{\!ij}p_{i,n}
+{\mathsf T}_{\!ji}p_{j,n}\\
\end{array}
\right.\\
\end{array}
\right.\\
\text{for every}\;i\in I\\
\left\lfloor
\begin{array}{l}
s_{i,n}=\gamma_n^{-1}(u_{i,n}-p_{i,n})+
\mu_n^{-1}Q_i\big(0,(l_{ij,n}-q_{ij,n})_{j\in J(i+)},
(l_{ji,n}-q_{ji,n})_{j\in J(i-)}\big)\\
\end{array}
\right.\\
\tau_n=\sum_{i\in I}\|s_{i,n}\|^2+\sum_{(i,j)\in K}
\|t_{ij,n}\|^2\\
\text{if}\;\tau_n=0\\
\left\lfloor
\begin{array}{l}
\theta_n=0\\
\end{array}
\right.\\
\text{if}\;\tau_n>0\\
\left\lfloor
\begin{array}{l}
\theta_n=\lambda_n\big(\gamma_n^{-1}\sum_{i\in I}
\|u_{i,n}-p_{i,n}\|^2+\mu_n^{-1}
\sum_{(i,j)\in K}\|l_{ij,n}-q_{ij,n}\|^2\big)/\tau_n\\
\end{array}
\right.\\
\text{for every}\;i\in I\\
\left\lfloor
\begin{array}{l}
u_{i,n+1/2}=u_{i,n}-\theta_n s_{i,n}\\
\text{for every}\;j\in J(i+)\\
\left\lfloor
\begin{array}{l}
g_{ij,n+1/2}=g_{ij,n}-\theta_n t_{ij,n}
\end{array}
\right.\\
\end{array}
\right.\\
\chi_n=\sum_{i\in I}\scal{u_{i,0}-u_{i,n}}{u_{i,n}-u_{i,n+1/2}}
+\sum_{(i,j)\in K}\scal{g_{ij,0}-g_{ij,n}}
{g_{ij,n}-g_{ij,n+1/2}}\\
\mu_n=\sum_{i\in I}\|u_{i,0}-u_{i,n}\|^2+\sum_{(i,j)\in K}
\|g_{ij,0}-g_{ij,n}\|^2\\
\nu_n=\sum_{i\in I}\|u_{i,n}-u_{i,n+1/2}\|^2+
\sum_{(i,j)\in K}\|g_{ij,n}-g_{ij,n+1/2}\|^2\\
\rho_n=\mu_n\nu_n-\chi_n^2\\
\text{if}\;\rho_n=0\;\text{and}\;\chi_n\geq 0\\
\left\lfloor
\begin{array}{l}
\text{for every}\;i\in I\\
\left\lfloor
\begin{array}{l}
u_{i,n+1}=u_{i,n+1/2}\\
\text{for every}\;j\in J(i+)\\
\left\lfloor
\begin{array}{l}
g_{ij,n+1}=g_{ij,n+1/2}
\end{array}
\right.\\
\end{array}
\right.\\
\end{array}
\right.\\
\text{if}\;\rho_n>0\;\text{and}\;\chi_n\nu_n\geq\rho_n\\
\left\lfloor
\begin{array}{l}
\text{for every}\;i\in I\\
\left\lfloor
\begin{array}{l}
u_{i,n+1}=u_{i,0}+(1+\chi_n/\nu_n)(u_{i,n+1/2}-u_{i,n})\\
\text{for every}\;j\in J(i+)\\
\left\lfloor
\begin{array}{l}
g_{ij,n+1}=g_{ij,0}+(1+\chi_n/\nu_n)(g_{ij,n+1/2}-g_{ij,n})
\end{array}
\right.\\
\end{array}
\right.\\
\end{array}
\right.\\
\text{if}\;\rho_n>0\;\text{and}\;\chi_n\nu_n<\rho_n\\
\left\lfloor
\begin{array}{l}
\text{for every}\;i\in I\\
\left\lfloor
\begin{array}{l}
u_{i,n+1}=u_{i,n}+(\nu_n/\rho_n)\big(\chi_n(u_{i,0}-u_{i,n})
+\mu_n(u_{i,n+1/2}-u_{i,n})\big)\\
\text{for every}\;j\in J(i+)\\
\left\lfloor
\begin{array}{l}
g_{ij,n+1}=g_{ij,n}+(\nu_n/\rho_n)
\big(\chi_n(g_{ij,0}-g_{ij,n})
+\mu_n(g_{ij,n+1/2}-g_{ij,n})\big).
\end{array}
\right.\\
\end{array}
\right.\\
\end{array}
\right.\\
\end{array}
\right.\\
\end{array}
\end{equation}
Then, for every $i\in I$ and $j\in J(i+)$, 
$u_{i,n}\stro{\HH_i}\overline{u}_i$ and
$g_{ij,n}\stro{L^2(\Upsilon_{ij})}\overline{g}_{ij}$.
\end{theorem}
\begin{proof}
Using \eqref{e:nypl2013-06-25a}, \eqref{e:erice2013-06-16x}, and 
\eqref{e:2014-07-02z}, it follows from  
\eqref{e:santiago2012-01-13} that \eqref{e:erice2013-06-16a} is 
a special case of \eqref{e:2014-06-05a}. In view of 
Proposition~\ref{p:0} and Theorem~\ref{t:5}, the proof is complete.
\end{proof}

\begin{remark}
\label{r:genova2014-07-08}
Algorithm \eqref{e:erice2013-06-16a} is mainly organized as a
series of loops indexed by the variables $i$ and $j$ that can 
be executed simultaneously and, therefore, implemented on 
parallel processors. The first loop computes $v_{i,n}$ as well as 
$p_{i,n}=\prox_{\gamma_n\varphi_i}v_{i,n}$ for each subdomain 
$i\in I$. The computation of $v_{i,n}$ involves the operator 
$Q_i$ which, in view of Problem~\ref{prob:1}\ref{A6}, amounts to 
solving the Dirichlet-Neumann boundary problem
\begin{equation}
\label{e:genova2014-07-07}
\begin{cases}
-\Delta u=0&\text{on}\;\Omega_i,\\
u=0&\text{on}\;\Upsilon_{ii},\\
\nu_i^\top D u=g_{ij,n}
&\text{on}\;\Upsilon_{ij},\:\:
\text{for every}\:\:j\in J(i+),\\
\nu_i^\top D u=-g_{ji,n}&\text{on}\;\Upsilon_{ij},\:\:
\text{for every}\:\:j\in J(i-).
\end{cases}
\end{equation} 
On the other hand, it follows from \eqref{e:prox} that 
\begin{equation}
p_{i,n}=
\displaystyle{\argmin{w\in\HH_i}
{\gamma_n\varphi_i(w)+\frac{1}{2}\int_{\Omega_i}
\big|Dw-Dv_{i,n}\big|^2}}.
\end{equation}
Likewise, the proximity operation across interface $\Upsilon_{ij}$ in 
the 
next loop is computed as
\begin{equation}
q_{ij,n}=\displaystyle{
\argmin{w\in L^2({\Upsilon_{ij}})}
{\mu_n\psi_{ij}(w)+
\frac12\int_{\Upsilon_{ij}}
\big|w-l_{ij,n}-\mu_ng_{ij,n}\big|^2dS}}.
\end{equation}
The remaining steps involve straightforward computations.
\end{remark}

\begin{remark}
The variational formulation of Problem~\ref{prob:1} can be modified
to include domain decomposition problems with overlapping 
subdomains. Indeed, for every $i\in I$ and $j\in J(i+)$, it is
necessary to consider a projection operator 
$P_{ij}\colon\HH_i\to H^1(\Omega_i\cap\Omega_j)$ instead of 
the trace operator $\mathsf{T}_{ij}\colon
\HH_i\to L^2(\Upsilon_{ij})$. An application of 
the overlapping framework to image processing with total variation 
and $\ell^1$ minimization can be found in \cite{Forn10}.
\end{remark}

\begin{remark}
An alternative approach in order to guarantee condition \ref{A8} in
Problem~\ref{prob:1} is to replace the Hilbert spaces 
$(L^2(\Upsilon_{ij}))_{(i,j)\in K}$ by
$(H^{1/2}(\Upsilon_{ij}))_{(i,j)\in K}$, in which case the trace 
operators are surjective \cite[Theorem~1.5.1.2]{Gris85}. The 
difficulty of this approach resides in the computation of the
proximity operators $(\prox_{\psi_{ij}})_{(i,j)\in K}$
in \eqref{e:erice2013-06-16a}, which is not easy because of the 
complexity of the metric of 
$(H^{1/2}(\Upsilon_{ij}))_{(i,j)\in K}$.
\end{remark}

\section{Special cases}
\label{sec:4}

We illustrate the potential use of algorithm 
\eqref{e:erice2013-06-16a} through a few applications 
to domain decomposition in the context of the Poisson, 
$p$--Laplacian, and obstacle problems with Dirichlet conditions
and continuity at the interfaces. 
We start with a couple of technical facts. First, define
\begin{equation}
\label{e:E_i}
(\forall i\in I)\quad E_i^p=\menge{u\in W^{1,p}(\Omega_i)}
{\mathsf{T}_{\!i}u=0\:\:\text{on }\Upsilon_{ii}}.
\end{equation}

\begin{proposition}
\label{p:gen}
Consider the setting of Problem~\ref{prob:1}. Let 
$p\in\left]1,\pinf\right[$, for every $i\in I$ let
$\phi_i\in\Gamma_0(W^{1,p}(\Omega_i))$ be a strictly convex coercive 
function with respect to the $W^{1,p}(\Omega_i)$ norm, and set 
\begin{equation}
\label{e:2014-07-09}
\varphi\colon W^{1,p}(\Omega)\to\RX\colon u\mapsto\sum_{i\in I}
\phi_i(u|_{\Omega_i}).
\end{equation} 
Then $\varphi$ is a strictly convex coercive function in
$\Gamma_0(W^{1,p}(\Omega))$ which is coercive 
with respect to the $W^{1,p}(\Omega)$ norm,
and the optimization problems
\begin{equation}
\label{e:gen1}
\minimize{u\in W_0^{1,p}(\Omega)}{\varphi(u)} 
\end{equation}
and 
\begin{equation}
\label{e:gen2}
\minimize{\substack{(u_i)_{i\in I}\in\bigtimes_{i\in I}E_i^p\\
(\forall (i,j)\in K)\:
\mathsf{T}_{ij}u_i=\mathsf{T}_{ji}u_j}}{\sum_{i\in I}\phi_i(u_i)} 
\end{equation}
have unique solutions $\overline{u}\in W_0^{1,p}(\Omega)$ and 
$(\overline{u}_i)_{i\in I}\in E_1^p\times\cdots\times E_m^p$, 
respectively. Moreover,
\begin{equation}
\label{e:defugen}
(\forall i \in I)\quad \overline{u}(x) = \overline{u}_i(x) \quad 
\text{for almost every } x\in \Omega_i.
\end{equation}
\end{proposition}
\begin{proof}
Let $u$ and $v$ be functions in $W^{1,p}(\Omega)$ such that 
$u\neq v$, and let $\alpha\in\zeroun$. There exists a measurable
set $U\subset\Omega$ of nonzero Lebesgue measure such that
$(\forall x\in U)$ $u(x)\neq v(x)$. For every $i\in I$, set 
$U_i=U\cap\Omega_i$. By assumption \ref{A2} in Problem~\ref{prob:1}, 
and the additivity property of the Lebesgue measure, there exists 
$j\in I$ such that $U_{j}$ has nonzero measure, which yields 
$u|_{\Omega_{j}}\neq v|_{\Omega_{j}}$. It then follows from the 
strict convexity of the functions $(\phi_i)_{i\in I}$ that
\begin{equation}
\sum_{i\in I}\phi_i\big((\alpha u+(1-\alpha)v)|_{\Omega_i}\big)
=\sum_{i\in I}\phi_i\big(\alpha 
u|_{\Omega_i}+(1-\alpha)v|_{\Omega_i}\big)
<\alpha\sum_{i\in I}\phi_i(u|_{\Omega_i})+(1-\alpha)
\sum_{i\in I}\phi_i(v|_{\Omega_i}),
\end{equation}
which shows that $\varphi$ is strictly convex. On the other hand, 
since assumption \ref{A2} in Problem~\ref{prob:1} yields, for every 
$u\in W^{1,p}(\Omega)$,
\begin{equation}
\|u\|_{W^{1,p}(\Omega)}^p=\int_{\Omega}|u|^p+\int_{\Omega}|Du|^p
=\sum_{i\in I}\int_{\Omega_i}|u|^p+\int_{\Omega_i}|Du|^p=\sum_{i\in 
I}\|u|_{\Omega_i}\|_{W^{1,p}(\Omega_i)}^p,
\end{equation}
the coercivity of $\varphi$ follows from the coercivity of 
the functions $(\phi_i)_{i\in I}$.

The existence of solutions $\overline{u}\in W_0^{1,p}(\Omega)$ and 
$(\overline{u}_i)_{i\in I}\in E_1^p\times\cdots\times E_m^p$,
respectively to \eqref{e:gen1} and \eqref{e:gen2},
follows from the classical theorems for the minimization of closed 
convex coercive functions on reflexive Banach spaces (see, e.g., 
\cite[Theorem~3.3.4]{Atto06}, \cite[Theorem~2.5.1(ii)]{Zali02}). 
The uniqueness is a consequence of the strict convexity of the 
objective functions. Set 
\begin{equation}
\label{e:tilde}
(\forall i\in I)\quad  \widetilde{u}(x)= \overline{u}_i(x)\quad
\text{for almost every }x\in\Omega_i.
\end{equation}
Since $\Omega\smallsetminus\bigcup_{i\in I}\Omega_i$ 
has zero Lebesgue measure, it follows from 
condition~\ref{A2} in Problem~\ref{prob:1}
that the function $\widetilde{u}$ is
well defined in $L^p(\Omega)$.
Let us prove that $\widetilde{u} = \overline{u}$, which will
complete the proof.
Arguing as in \cite[Lemma~6.4.1]{Atto06}, we deduce that, 
for every $u\in L^p(\Omega)$,
\begin{equation}
\label{jumpcond1gen}
u\in W^{1,p}(\Omega)\quad\Leftrightarrow\quad(\forall (i,j)\in K)
\quad u|_{\Omega_i}\in W^{1,p}(\Omega_i)\:\:
\text{and}\:\:\mathsf{T}_{\!ij}(u|_{\Omega_i})=
\mathsf{T}_{\!ji}(u|_{\Omega_j}).
\end{equation}
The characterization \eqref{jumpcond1gen} expresses the fact that 
the jumps of every $u\in W^{1,p}(\Omega)$ across the interfaces 
$(\Upsilon_{ij})_{(i,j)\in K}$ are zero.
Correspondingly, taking into account the Dirichlet boundary 
condition \cite[Section~2.1]{Ekel74}, we deduce that, for every 
$u\in L^p(\Omega)$,
\begin{equation}
\label{jumpcondgen}
u\in W^{1,p}_0(\Omega)\quad\Leftrightarrow\quad(\forall (i,j)\in K)
\quad u|_{\Omega_i}\in E_i^p\:\:\text{and}\:\:
\mathsf{T}_{\!ij}(u|_{\Omega_i})
=\mathsf{T}_{\!ji}(u|_{\Omega_j}).
\end{equation}
It then follows from \eqref{e:tilde} that, for every $i\in I$, 
$\widetilde{u}|_{\Omega_i}=\overline{u}_i\in E_i^p$, and, for
every $(i,j)\in K$, $\mathsf{T}_{\!ij}(\widetilde{u}|_{\Omega_i})
=\mathsf{T}_{\!ij}\overline{u}_i=\mathsf{T}_{\!ji}\overline{u}_j=
\mathsf{T}_{\!ji}(\widetilde{u}|_{\Omega_j})$. Hence,
\eqref{jumpcondgen}
yields $\widetilde{u}\in W^{1,p}_0(\Omega)$ and, for every 
$u\in W^{1,p}_0(\Omega)$, \eqref{e:2014-07-09} yields (the sets 
$(\Omega_i)_{i\in I}$ are disjoint, and the Lebesgue measure of 
the interfaces is zero)
\begin{equation}
\varphi(\widetilde{u})=\sum_{i\in I}
\phi_i(\overline{u}_i)\leq\sum_{i\in I}
\phi_i(u|_{\Omega_i})=\varphi(u),
\end{equation}
which, by uniqueness of the solution, yields
$\widetilde{u}=\overline{u}$.
\end{proof}

\begin{proposition}
\label{p:13}
Consider the setting of Problem~\ref{prob:1}. Let 
$\gamma\in\RPP$, let $f\in L^2(\Omega)$, and, for every $i\in I$, 
let $C_i$ be a nonempty closed convex subset of $\HH_i$.
Suppose that
\begin{equation}
\label{e:varphi_i}
\varphi_i\colon\HH_i\to\RX\colon u_i\mapsto\iota_{C_i}(u_i)+\Frac12
\int_{\Omega_i}|Du_i|^2-\int_{\Omega_i}fu_i.
\end{equation}
Then the following hold for every $i\in I$:
\begin{enumerate}
\item 
We have
\label{p:13i} 
\begin{equation}
\begin{cases}
\varphi_i\colon u_i\mapsto\iota_{C_i}(u_i)+\Frac12\|u_i\|^2-
\scal{Q_i(f,0,\ldots,0)}{u_i}\\
\partial\varphi_i =N_{C_i}+\Id-Q_i(f,0,\ldots,0)\\
\prox_{\gamma\varphi_i}=P_{C_i}\bigg(\Frac{1}{1+\gamma}\Id+
\Frac{\gamma}{1+\gamma}Q_i(f,0,\ldots,0)\bigg).
\end{cases}
\end{equation}
\item
\label{p:13ii} 
Suppose that $C_i=\HH_i$. Then $\varphi_i$ is 
G\^ateaux--differentiable on $\HH_i$ and
\begin{equation}
\begin{cases}
\label{e:l1ii}
\varphi_i\colon u_i\mapsto\Frac12\|u_i\|^2-
\scal{Q_i(f,0,\ldots,0)}{u_i}\\
\nabla\varphi_i=\Id-Q_i(f,0,\ldots,0)\\
\prox_{\gamma\varphi_i}=\Frac{1}{1+\gamma}\Id+
\Frac{\gamma}{1+\gamma}Q_i(f,0,\ldots,0).
\end{cases}
\end{equation}
\end{enumerate}
\end{proposition}
\begin{proof}
Fix $i\in I$.
First note that 
\begin{equation}
\phi_i\colon\HH_i\to\RR\colon u_i\mapsto\int_{\Omega_i}fu_i 
\end{equation}
is linear. Moreover, since $\Omega_i$ bounded, 
the Cauchy-Schwarz and Poincar\'e's inequalities 
\cite[Appendix~(53c)]{Zeid90B}, and 
\eqref{e:scaledp} yield
\begin{equation}
(\exi\delta\in\RPP)(\forall u_i\in\HH_i)\quad
|\phi_i(u_i)|\leq\|f\|_{L^2(\Omega_i)}
\|u_i\|_{L^2(\Omega_i)}\leq \delta\|f\|_{L^2(\Omega_i)}\|u_i\|.
\end{equation}
Hence, the Riesz-Fr\'echet representation theorem 
asserts that there exists a unique $v_i\in\HH_i$ such that
\begin{equation}
(\forall u_i\in\HH_i)\quad\phi_i(u_i)=\int_{\Omega_i}fu_i
=\int_{\Omega_i} (D v_i)^\top D u_i=\scal{v_i}{u_i}.
\end{equation}
Thus, it follows from \cite[Proposition~25.28]{Zeid90B} and 
\eqref{e:Q_i} 
that $v_i=Q_i(f,0,\ldots,0)$. Using \eqref{e:scaledp}, we can 
therefore write \eqref{e:varphi_i} as
\begin{equation}
\label{e:phiiriesz}
\varphi_i\colon u_i\mapsto\Frac12\|u_i\|^2-
\scal{Q_i(f,0,\ldots,0)}{u_i}+\iota_{C_i}(u_i). 
\end{equation}
Moreover, we deduce from standard subdifferential calculus 
\cite[Section~16.4]{Livre1} that
\begin{equation}
\label{e:partialphi}
\partial\varphi_i=\Id-Q_i(f,0,\ldots,0)+N_{C_i},
\end{equation}
where $N_{C_i}$ is the normal cone operator to $C_i$.
Hence, it follows from \eqref{e:partialphi} that, for 
every $u_i$ and $p_i$ in $\HH_i$,
\begin{eqnarray}
p_i=\prox_{\gamma\varphi_i}u_i
&\Leftrightarrow& u_i-p_i\in\gamma\partial\varphi_i(p_i)
\nonumber\\
&\Leftrightarrow& u_i\in (1+\gamma)p_i-\gamma Q_i(f,0,\ldots,0)
+N_{C_i}p_i\nonumber\\
&\Leftrightarrow& \Frac{1}{1+\gamma}u_i+
\Frac{\gamma}{1+\gamma}Q_i(f,0,\ldots,0)
\in p_i+N_{C_i}p_i\nonumber\\
&\Leftrightarrow& p_i=P_{C_i}\bigg(\Frac{1}{1+\gamma}u_i+
\Frac{\gamma}{1+\gamma}Q_i(f,0,\ldots,0)\bigg).
\end{eqnarray}

\ref{p:13ii}: Since 
$N_{C_i}\equiv\{0\}$ and $P_{C_i}=\Id$, the result follows 
from \ref{p:13i}.
\end{proof}

\subsection{Poisson problem}
\label{sub-Poisson}

Let $f\in L^2(\Omega)$, and consider the Poisson problem with an 
homogeneous Dirichlet boundary condition
\begin{equation}
\label{Poi1}
\begin{cases}
-\Delta u=f,&\text{on}\:\:\Omega;\\
u=0,&\text{on}\:\:\bdry\Omega.
\end{cases}
\end{equation}
Classically, this problem has a unique weak solution 
$\overline{u}\in H_0^1(\Omega)$,
which can be obtained by solving the strongly convex minimization 
problem (see 
\cite[Chapter~IV.2.1]{Ekel74} or \cite[Chapter~25.9]{Zeid90B})
\begin{equation}
\label{Dir1}
\minimize{u\in H_0^1(\Omega)}{\Frac12\int_{\Omega}|Du|^2-
\int_{\Omega}fu}.
\end{equation}
As a simple example of the flexibility of our framework, we  
solve \eqref{Dir1} by decomposing the domain 
$\Omega$ into subdomains satisfying the hypotheses in 
Problem~\ref{prob:1}, and 
by imposing continuity conditions at the interfaces. 

\begin{problem}
\label{prob:Poisson}
Consider the setting of Problem~\ref{prob:1}. Let 
$f\in L^2(\Omega)$ and, for every $(i,j)\in K$, assume that 
$\Upsilon_{ij}$ and $\bdry\Omega$
are of class ${\EuScript C}^2$. The problem is to 
\begin{equation}
\label{e:erice2013-06-16poisP}
\minimize{\substack{(u_i)_{i\in I}\in\bigoplus_{i\in I}\HH_i\\
(\forall (i,j)\in K)\:
\mathsf{T}_{ij}u_i=\mathsf{T}_{ji}u_j}}{ 
\sum_{i=1}^m\Frac12\int_{\Omega_i}|Du_i|^2-\int_{\Omega_i}fu_i}.
\end{equation}
\end{problem}

We first show the equivalence between Problem~\ref{prob:Poisson} 
and \eqref{Dir1}. 

\begin{proposition}
\label{p:poisson1}
The optimization problem in \eqref{e:erice2013-06-16poisP} has 
a unique solution 
$(\overline{u}_i)_{i\in I}$. Moreover, the function defined in
\eqref{e:defugen}
is the unique solution to \eqref{Dir1}. 
\end{proposition}
\begin{proof}
This is a consequence of Proposition~\ref{p:gen} with $p=2$ and, 
for every $i\in I$, 
$\phi_i\colon u\mapsto\frac12\int_{\Omega_i}|Du|^2-
\int_{\Omega_i}fu$, which are strongly convex. In this case 
$\varphi\colon u\mapsto\frac12\int_{\Omega}|Du|^2-
\int_{\Omega}fu$.
\end{proof}

Our method for solving Problem~\ref{prob:Poisson} is a particular 
case of \eqref{e:erice2013-06-16a}. Hence, the following
convergence result is an application of Theorem~\ref{t:1}.

\begin{theorem}
\label{t:poisson}
In algorithm \eqref{e:erice2013-06-16a} of Theorem~\ref{t:1}, 
replace the steps defining $p_{i,n}$ and $q_{ij,n}$ by
\begin{equation}
\label{e:proxpoi}
p_{i,n}=\Frac{1}{1+\gamma_n}v_{i,n}+
\Frac{\gamma_n}{1+\gamma_n}Q_i(f,0,\ldots,0)\quad\text{and}\quad
q_{ij,n}=0,
\end{equation}
respectively. Then, for every $i\in I$, the 
sequence $(u_{i,n})_{n\in\NN}$ generated by 
\eqref{e:erice2013-06-16a} 
converges strongly to $\overline{u}_i$ in $\HH_i$.
\end{theorem}
\begin{proof}
Set
\begin{equation}
\label{e:deffct}
\begin{cases}
(\forall i\in I)\quad \displaystyle{\varphi_i\colon x\mapsto\Frac12
\int_{\Omega_i}|Du_i|^2-
\int_{\Omega_i}fu_i}\\
(\forall (i,j)\in K)\quad \psi_{ij}=\iota_{\{0\}}.
\end{cases}
\end{equation}
Since, for every $(i,j)\in K$,
$\varphi_i\in\Gamma_0(\HH_i)$ and 
$\psi_{ij}\in\Gamma_0 (L^2(\Upsilon_{ij}))$, 
Problem~\ref{prob:Poisson} is a particular 
case of \eqref{e:erice2013-06-16p}. 
Let us verify that condition \eqref{e:exicond} holds.
Let $(\overline{u}_i)_{i\in I}\in\HH_1\oplus\cdots\oplus\HH_m$ 
be the solution to \eqref{e:erice2013-06-16poisP} guaranteed by
Proposition~\ref{p:1} and let $\overline{u}\in H^1_0(\Omega)$
be as in \eqref{e:defugen}. Since $\psi_{ij}=
\iota_{\{0\}}$, we have $\partial\psi_{ij}(0)=L^2(\Upsilon_{ij})$ 
and, hence, the first condition in \eqref{e:exicond} is 
satisfied. 
Since $\bdry\Omega$ and $(\Upsilon_{ij})_{(i,j)\in K}$ are 
of class ${\EuScript C}^2$, \cite[Theorem~2.2.2.3]{Gris85} yields 
$\overline{u}\in H^2(\Omega)$. 
Therefore, we deduce from \cite[Theorem~1.5.1.2]{Gris85} that, 
for every $i\in I$ and $j\in J(i)$, $\nu_i^\top D\overline{u}_i$ 
and $\nu_j^\top D\overline{u}_j$ belong to $L^2(\Upsilon_{ij})$. 
Now let us show that the second condition in 
\eqref{e:exicond} holds with 
\begin{equation}
\label{e:gij}
(\forall (i,j)\in K)\quad
\overline{g}_{ij}=\nu_j^\top D\overline{u}_j\in L^2(\Upsilon_{ij}).
\end{equation}
We note that the solution $(\overline{u}_i)_{i\in I}$ to 
Problem~\ref{prob:Poisson} satisfies (see, e.g., 
\cite[Theorem~6.4.1]{Atto06})
\begin{align}
\label{transcond2}
(\forall i\in I)\quad
\begin{cases}
-\Delta\overline{u}_i=f,&\text{on}\;\Omega_i;\\
\overline{u}_i=0, &\text{on}\;\Upsilon_{ii};\\
{\mathsf T}_{ij}\,\overline{u}_i={\mathsf T}_{ji}\overline{u}_j,
&\text{on}\;\Upsilon_{ij},\:\:\text{for every}\:\:j\in J(i);\\
\nu_i^\top D\overline{u}_i=-\nu_j^\top D \overline{u}_j,
&\text{on}\;\Upsilon_{ij},\:\:
\text{for every}\:\:j\in J(i)
\end{cases}
\end{align}
in the sense of distributions,
which, from \eqref{e:Q_i}, yields 
\begin{equation}
(\forall i\in I)\quad \overline{u}_i=Q_i(f,(-\nu_j^\top D 
\overline{u}_j)_{j\in J(i+)},
(\nu_j^\top D \overline{u}_j)_{j\in J(i-)}).
\end{equation}
Let us observe that, because of the regularity $\overline{u}\in 
H^2(\Omega)$, the transmission conditions 
satisfied by $\overline{u}$ 
can be expressed as equalities in the spaces $L^2(\Upsilon_{ij})$, 
which fits in our abstract framework.
Since, for every $(i,j)\in K$, 
$\nu_i^\top D\overline{u}_i=-\nu_j^\top D\overline{u}_j$,
\eqref{e:gij} implies that
\begin{equation}
\overline{u}_i=Q_i\big(f,(-\nu_j^\top D \overline{u}_j)_{j\in J(i+)},
(\nu_j^\top D \overline{u}_j)_{j\in J(i-)}\big)
=Q_i\big(f,(-\overline{g}_{ij})_{j\in J(i+)},
(-\overline{g}_{ji})_{j\in J(i-)}\big).
\end{equation}
Hence, upon invoking Proposition~\ref{p:13}\ref{p:13ii} and the 
linearity of $Q_i$, we obtain
\begin{align}
\nabla\varphi_i(\overline{u}_i)
&=\overline{u}_i-Q_i(f,0,\ldots,0)\nonumber\\
&=Q_i\big(f,(-\overline{g}_{ij})_{j\in J(i+)},
(-\overline{g}_{ji})_{j\in J(i-)}\big)
-Q_i(f,0,\ldots,0)\nonumber\\
&=Q_i\big(0,(-\overline{g}_{ij})_{j\in J(i+)},
(-\overline{g}_{ji})_{j\in J(i-)}\big)\nonumber\\
&=-Q_i\big(0,(\overline{g}_{ij})_{j\in J(i+)},
(\overline{g}_{ji})_{j\in J(i-)}\big),
\end{align}
which is the second condition in \eqref{e:exicond}. On the other 
hand, it follows from \eqref{e:prox} and \eqref{e:deffct} that, 
for every $(i,j)\in K$ and every $\mu\in\RPP$, 
$\prox_{\mu\psi_{ij}}\equiv0$. Hence, we deduce from 
Proposition~\ref{p:13}\ref{p:13ii} that \eqref{e:proxpoi}
yields
\begin{equation}
\label{e:proxpoi2}
(\forall n\in\NN)\quad 
\begin{cases}
(\forall i\in I)\quad 
p_{i,n}=\prox_{\gamma_n\varphi_i}v_{i,n}\\
(\forall (i,j)\in K)\quad 
q_{ij,n}=\prox_{\mu_n\psi_{ij}}(l_{ij,n}+\mu_ng_{ij,n}),
\end{cases}
\end{equation}
and the result follows from 
Theorem~\ref{t:1} with $(\varphi_i)_{i\in I}$ and
$(\psi_{ij})_{(i,j)\in K}$ defined as in \eqref{e:deffct}.
\end{proof}

\begin{remark}\
\begin{enumerate}
\item
Note that $(\overline{g}_{ij})_{(i,j)\in K}$ defined in 
\eqref{e:gij} is a solution to the dual problem associated with 
Problem~\ref{prob:Poisson}. The method proposed in 
Theorem~\ref{t:poisson} also converge in the dual variables, but 
for the sake of simplicity we provide only the convergence in primal 
variables.
\item In \eqref{e:erice2013-06-16a} we have
\begin{equation}
(\forall n\in\NN)(\forall i\in I)\quad v_{i,n}=u_{i,n}-\gamma_nQ_i
\big(0,(g_{ij,n})_{j\in J(i+)},(g_{ji,n})_{j\in J(i-)}\big).
\end{equation}
Hence, since the operators $(Q_i)_{i\in I}$ defined in \eqref{e:Q_i} 
are multilinear, the sequences $(p_{i,n})_{i\in I,\, n\in\NN}$ can 
be computed more efficiently via
\begin{equation}
(\forall n\in\NN)
(\forall i\in I)\quad 
p_{i,n}=\Frac{1}{1+\gamma_n}u_{i,n}+
\Frac{\gamma_n}{1+\gamma_n}Q_i\big(f,(-g_{ij,n})_{j\in 
J(i+)},(-g_{ji,n})_{j\in J(i-)}\big).
\end{equation}
This allows us to solve only $m$ auxiliary PDE's 
for updating $(p_{i,n})_{i\in I}$ at each iteration $n$.
\end{enumerate}
\end{remark}

\begin{remark}
\label{elastic} 
The analysis of Theorem~\ref{t:poisson} can be adapted 
to the case of the linear elasticity system by using
Korn's inequality instead of Poincar\'e's inequality. A key
ingredient (and possible limitation) of our approach is the $H^2$
regularity property of the solution of the problem in the case of
the linear elasticity system.
Likewise fluid-solid interactions can be handled via our framework.
\end{remark}

\subsection{\textit{p}-Laplacian} 

It has long been observed that semi-linear and quasi-linear monotone 
problems can be efficiently analyzed using modern convex-analytical
tools \cite{AD,Br1,Zeid90A}. 
We follow a similar approach in applying our variational 
decomposition method to the $p$-Laplacian operator $\Delta_p$.

Let $p\in\left]1,\pinf\right[$, let $f\in L^{\infty}(\Omega)$,
and consider the partial differential equation governed by the 
$p$-Laplacian operator with Dirichlet boundary conditions 
\begin{equation}
\label{pLap1}
\begin{cases}
-\mbox{div}\left(|Du|^{p-2}Du\right)=f,&\text{on}\:\:\Omega;\\
u=0,&\text{on}\:\:\bdry\Omega.
\end{cases}
\end{equation}
Note that, if $p=2$, \eqref{pLap1} reduces to \eqref{Poi1}. This 
problem possesses a unique weak solution $\overline{u}\in W_0^{1,p}
(\Omega)$, which can be obtained by solving the strictly convex 
minimization problem \cite[Section~IV.2.2]{Ekel74}
\begin{equation}
\label{pLap2}
\minimize{u\in W_0^{1,p}(\Omega)}{\Frac{1}{p}\int_{\Omega}|Du|^p-
\int_{\Omega}fu}.
\end{equation}
As another example of our framework, we are interested 
to solve \eqref{pLap2} by decomposing the domain 
$\Omega$ in subdomains satisfying the hypotheses in 
Problem~\ref{prob:1}, and  considering continuity conditions on 
the interfaces. More precisely, we are interested in the following
problem.

\begin{problem}
\label{prob:plap}
Consider the setting of Problem~\ref{prob:1}. Let 
$p\in\left]1,\pinf\right[$, let $\alpha\in\left]0,1\right]$, and let 
$f\in L^{\infty}(\Omega)$. 
Suppose that the unique solution to \eqref{pLap2} is in
$\EuScript{C}^{1,\alpha}(\overline{\Omega})$. 
The problem is to 
\begin{equation}
\label{Dirp}
\minimize{\substack{(u_i)_{i\in I}\in\bigtimes_{i\in I} E_i^p\\
(\forall (i,j)\in K)\:
\mathsf{T}_{ij}u_i=\mathsf{T}_{ji}u_j}}{ 
\sum_{i=1}^m\Frac1p\int_{\Omega_i}|Du_i|^p-\int_{\Omega_i}fu_i}.
\end{equation}
\end{problem}

\begin{proposition}
\label{p:2}
Problem~\ref{prob:plap} has a unique solution 
$(\overline{u}_i)_{i\in I}$. Moreover, the function 
$\overline{u}$ defined in \eqref{e:defugen}
is the unique solution to \eqref{pLap2}. 
\end{proposition}
\begin{proof}
This is a consequence of Proposition~\ref{p:gen} where, 
for every $i\in I$, 
$\phi_i\colon u\mapsto\frac1p\int_{\Omega_i}|Du|^p-
\int_{\Omega_i}fu$, which is strictly convex and coercive. 
In this case $\phi\colon u\mapsto\frac1p\int_{\Omega}|Du|^p-
\int_{\Omega}fu$.
\end{proof}

We now present our method for solving Problem~\ref{prob:plap}.

\begin{theorem}
\label{t:plap}
In algorithm \eqref{e:erice2013-06-16a} of Theorem~\ref{t:1}, 
replace the steps defining $p_{i,n}$ and $q_{ij,n}$ by 
\begin{equation}
\label{e:proxplap}
p_{i,n}=\argmin{w\in\HH_i\cap 
E_i^p}{\gamma_n\bigg(\Frac{1}{p}\int_{\Omega_i}|Dw|^p
-\int_{\Omega_i}fw\bigg)+\Frac{1}{2}\int_{\Omega_i}|Dw-Dv_{i,n}|^2}
\quad\text{and}\quad
q_{ij,n}=0,
\end{equation}
respectively. Then, for every $i\in I$, the 
sequence $(u_{i,n})_{n\in\NN}$ generated by 
\eqref{e:erice2013-06-16a} 
converges strongly to $\overline{u}_i$ in $\HH_i$.
\end{theorem}
\begin{proof}
We consider two cases.

(a) $p\geq 2$: Since $\Omega$ is bounded, we have 
$W^{1,p}(\Omega)\subset H^1(\Omega)$, and hence it follows from 
\eqref{e:E_i} that $E_i^p\subset\HH_i$. 
Thus, Problem~\ref{prob:plap} corresponds to the 
special case of Problem~\ref{prob:1} in which
\begin{equation}
\label{e:deffctpplap}
\begin{cases}
(\forall i\in I)\quad
\displaystyle{\varphi_i\colon\HH_i\to\RX\colon u_i\mapsto
\begin{cases}
\displaystyle{\Frac{1}{p}\int_{\Omega_i}|Du_i|^p-
\int_{\Omega_i} fu_i},
&\text{if}\:\:u_i\in E_i^p;\\  
+\infty, &\text{otherwise}
\end{cases}}\\
(\forall (i,j)\in K)\quad \psi_{ij}=\iota_{\{0\}}.
\end{cases}
\end{equation}
It is clear that the functions $(\psi_{ij})_{(i,j)\in K}$
are proper, lower semicontinuous, and convex. Since the convexity
of functions $(\varphi_i)_{i\in I}$ is clear, let us show that they 
are lower semicontinuous. To this end, fix $i\in I$, take 
$\lambda\in\RR$, and let $(u_n)_{n\in\NN}$ be a sequence in $\HH_i$
such that $u_n\stro{\HH_i} u\in\HH_i$ and
$(\forall n\in\NN)$ $\varphi_i(u_n)\leq\lambda$.
We deduce from \cite[Theorem~5.4.3]{Atto06} that the norm in 
$W^{1,p}(\Omega_i)$ and the norm
\begin{equation}
\label{e:poinc2}
u\mapsto\bigg(\int_{\Omega_i}|Du|^p\bigg)^{1/p}
=\|Du\|_{L^p(\Omega_i)} 
\end{equation}
are equivalent in $E_i^p$, which yields the coercivity of 
$\varphi_i$ in $E_i^p$. Therefore, $(u_n)_{n\in\NN}$ is 
bounded in $E_i^p$ and, hence, it converges weakly to $u$ in 
$E_i^p$. Moreover, the function $\varphi_i$ is convex 
and continuous on $E_i^p$, and hence weakly lower 
semicontinuous, which yields 
\begin{equation}
\varphi_i(u)\leq\varliminf\varphi_i(u_n)\leq\lambda.
\end{equation}
Let us show that condition \eqref{e:exicond} holds. 
Let $(\overline{u}_i)_{i\in I}\in\HH_1\oplus\cdots\oplus\HH_m$ 
be the solution to Problem~\ref{prob:plap}, and let $\overline{u}\in 
H^1_0(\Omega)$ be as in \eqref{e:defugen}. Since $\psi_{ij}=
\iota_{\{0\}}$, we have $\partial\psi_{ij}(0)=L^2(\Upsilon_{ij})$, 
and the first condition in \eqref{e:exicond} is therefore 
satisfied. Now since 
$\overline{u}\in\EuScript{C}^{1,\alpha}(\overline{\Omega})$, 
for every $(i,j)\in K$, $\nu_i^\top
|D{\overline{u}}_i|^{p-2}D{\overline{u}}_i\in 
L^2(\Upsilon_{ij})$ and $\nu_j^\top
|D{\overline{u}}_j|^{p-2}D{\overline{u}}_j\in 
L^2(\Upsilon_{ij})$. Let us show that the second condition in 
\eqref{e:exicond} holds with 
\begin{equation}
\label{e:gijplap}
(\forall (i,j)\in K)\quad
\overline{g}_{ij}=|D\overline{u}_j|^{p-2}\nu_j^\top 
D\overline{u}_j\in 
L^2(\Upsilon_{ij}).
\end{equation}
The Euler equation associated with Problem~\ref{prob:plap}
yields 
\begin{align}
\label{pLap5}
(\forall i\in I)\quad
\begin{cases}
-\mbox{div}\left(|D{\overline{u}}_i|^{p-2}D{\overline{u}}_i\right)
=f,&\text{on}\;\Omega_i;\\
\overline{u}_i=0, &\text{on}\;\Upsilon_{ii};\\
{\mathsf T}_{ij}\,\overline{u}_i={\mathsf T}_{ji}\overline{u}_j,
&\text{on}\;\Upsilon_{ij},\:\:\text{for every}\:\:j\in J(i);\\
|D\overline{u}_i|^{p-2}\nu_i^\top D\overline{u}_i
=-|D\overline{u}_j|^{p-2}\nu_j^\top D \overline{u}_j,
&\text{on}\;\Upsilon_{ij},\:\:
\text{for every}\:\:j\in J(i).
\end{cases}
\end{align}
Now, for every $i\in I$, let us compute an element 
$v_i\in\partial{\varphi}_i(\overline{u}_i)$. 
By a classical directional differentiation argument (see 
\cite[Theorem 6.6.1]{Atto06} for a detailed proof) we obtain 
\begin{equation}
(\forall u\in\HH_i)\quad\int_{\Omega_i}(|D\overline{u}_i|^{p-2} 
D\overline{u}_i-Dv_i)^\top Du=\int_{\Omega_i}fu, 
\end{equation}
from which we deduce that $v_i$ satisfies, in sense of 
distributions, the boundary value problem 
\begin{equation}
\label{pLap6}
\begin{cases}
-\Delta v_i=-f-\mbox{div}\left(|D{\overline{u}}_i|^{p-2}
D{\overline{u}}_i \right)\!,&\text{on}\;\Omega_i;\\
v_i=0,&\text{on}\;\Upsilon_{ii};\\
\nu_i^\top Dv_i=\nu_i^\top |D{\overline{u}}_i|^{p-2}
D{\overline{u}}_i,&\text{on}\;\Upsilon_{ij},\:\:
\text{for every}\:\:j\in J(i),
\end{cases}
\end{equation}
which, using \eqref{pLap5} and \eqref{e:gijplap}, reduces to
\begin{equation}
\label{pLap7}
\begin{cases}
\Delta v_i=0,&\text{on}\;\Omega_i;\\
v_i=0,&\text{on}\;\Upsilon_{ii};\\
\nu_i^\top Dv_i=-\overline{g}_{ij},&\text{on}\;\Upsilon_{ij},\:\:
\text{for every}\:\:j\in J(i+);\\
\nu_i^\top Dv_i= \overline{g}_{ji},&\text{on}\;\Upsilon_{ij},\:\:
\text{for every}\:\:j\in J(i-).
\end{cases}
\end{equation}
Hence, we derive from \eqref{e:Q_i} that 
$v_i=Q_i(0,(-\overline{g}_{ij})_{j\in 
J(i+)},(-\overline{g}_{ji})_{j\in 
J(i-)})\in\partial\varphi_i(\overline{u}_i)$ which yields 
\eqref{e:exicond}.
On the other hand, it follows from \eqref{e:prox} and 
\eqref{e:deffct} that, for every $(i,j)\in K$ and every 
$\mu\in\RPP$, $\prox_{\mu\psi_{ij}}\equiv0$. Hence, we deduce from 
\eqref{e:prox} that \eqref{e:proxplap}
yields
\begin{equation}
\label{e:proxplap2}
(\forall n\in\NN)\quad 
\begin{cases}
(\forall i\in I)\quad 
p_{i,n}=\prox_{\gamma_n\varphi_i}v_{i,n}\\
(\forall (i,j)\in K)\quad 
q_{ij,n}=\prox_{\mu_n\psi_{ij}}(l_{ij,n}+\mu_ng_{ij,n}).
\end{cases}
\end{equation}
Therefore, when $(\varphi_i)_{i\in I}$ and
$(\psi_{ij})_{(i,j)\in K}$ are defined by 
\eqref{e:deffctpplap}, we deduce from 
Theorem~\ref{t:1} that $u_{i,n}\stro{\HH_i}\overline{u}_i$.

\if
In addition, since, for every $i\in I$, 
the operator $u\mapsto\int_{\Omega_i}fu$ is linear and bounded, we 
obtain
\begin{equation}
\label{e:convlin}
(\forall i\in I)\quad \int_{\Omega_i}fu_{i,n}\to
\int_{\Omega_i}f\overline{u}_i,
\end{equation}
and, since, for every $i\in I$, the norm of $W^{1,p}(\Omega_i)$ is 
equivalent 
to the norm defined in \eqref{e:poinc2} in $E_i^p$, we deduce from 
\eqref{e:deffctpplap} and \eqref{e:convlin} that
\begin{equation}
\label{e:convforte1}
\sum_{i=1}^m\|u_{i,n}\|_{W^{1,p}(\Omega_i)}^p\to 
\sum_{i=1}^m\|\overline{u}_i\|_{W^{1,p}(\Omega_i)}^p.
\end{equation}
Thus, for every $i\in I$, the sequence $(u_{i,n})_{n\in\NN}$
is bounded in $W^{1,p}(\Omega_i)$. Therefore, since 
$u_{i,n}\stro{\HH_i}\overline{u}_i$, we have that a subsequence of 
$(u_{i,n})_{n\in\NN}$
converges weakly to $\overline{u}_i$ in $W^{1,p}(\Omega_i)$, and 
hence,
by proceeding similarly, we obtain $u_{i,n}\weakly\overline{u}_i$ in
$W^{1,p}(\Omega_i)$. Finally, we deduce from \eqref{e:convforte1} 
and
\cite[Proposition~2.4.10]{Atto06} that, for every $i\in I$, 
$u_{i,n}\to\overline{u}_i$ in $W^{1,p}(\Omega_i)$.

\fi

(b) $1< p <2$: In this case, for every $i\in I$, 
$\HH_i\subset W^{1,p}(\Omega_i)$, with continuous embedding. 
Let us assume that the solution $\overline{u}$ of problem 
\eqref{pLap2} belongs to $H^1_0 (\Omega)$ (indeed we shall further 
state regularity properties of $\overline{u}$ which make this 
property satisfied). Combining this property with the density of 
$H^1_0 (\Omega)$ in $W^{1,p}_0 (\Omega)$ (for the norm topology 
of $W^{1,p}_0 (\Omega)$), the variational problem \eqref{pLap2} 
equivalently writes
\begin{equation}
\label{pLap21}
\minimize{u\in H^1_0 (\Omega)}{\Frac{1}{p}\int_{\Omega}|Du|^p-
\int_{\Omega}fu}.
\end{equation}
Using the same argument as in Proposition~\ref{p:1}, 
this is equivalent to solving
\begin{equation}
\label{Dirp2}
\minimize{\substack{(u_i)_{i\in I}\in\bigoplus_{i\in I}\HH_i\\
(\forall (i,j)\in K)\:
\mathsf{T}_{ij}u_i=\mathsf{T}_{ji}u_j}}{ 
\sum_{i=1}^m\Frac1p\int_{\Omega_i}|Du_i|^p-\int_{\Omega_i}fu_i}.
\end{equation}
Thus we are led to set
\begin{equation}
(\forall i\in I)\quad
\varphi_i\colon\HH_i\to\RR\colon u_i\mapsto\Frac{1}{p} 
\int_{\Omega_i}|Du_i|^p-\int_{\Omega_i}fu_i,          
\end{equation}
which is continuous on $\HH_i$.
The remainder of the proof is identical to the case $p\geq 2$. 
Just notice that, when $p<2$, the $p$-Laplacian becomes a singular 
elliptic operator. The global regularity of the solution 
$\overline{u}$ to problem \eqref{pLap2}, with a globally continuous 
gradient, is well established \cite{BdV,LB}. 
\end{proof}

\begin{remark}\
\label{rem:1}
\begin{enumerate}
\item 
A recent account of regularity properties for the solution 
to the $p$-Laplacian equation can be found in \cite{BdV,Lib,Tolk}. 
Note that, in contrast with the case $p=2$, the 
degeneracy of the elliptic operator $-{\Delta}_p$ for $p>2$ 
makes the regularity study more involved. In \cite{BdV}, global 
$H^2(\Omega)$ regularity is obtained for the regularized operator 
$-\varepsilon \Delta -\Delta_p$ ($\varepsilon >0$). In general,
for smooth data, the local regularity $C^{1,\alpha}_{loc}(\Omega)$ 
holds ($\alpha\in\RPP$). 
\item 
Our approach makes it possible to consider the case when $p$
assumes different values on each subdomain $\Omega_i$. In this
case,  the minimization problem becomes
\begin{equation}
\label{piLap}
\minimize{\substack{(u_i)_{i\in I}\in\bigoplus_{i\in I}\HH_i\\
(\forall (i,j)\in K)\:
\mathsf{T}_{ij}u_i=\mathsf{T}_{ji}u_j}}{ 
\sum_{i=1}^m\varphi_i (u_i)}
\end{equation}
where, for every $i\in I$, 
\begin{align}
\varphi_i\colon\HH_i\to\RX\colon u_i\mapsto 
\begin{cases}
\displaystyle{\Frac{1}{p_i}\int_{\Omega_i}|Du_i|^{p_i}-
\int_{\Omega_i}fu_i},&\text{if}\:\:u_i\in E_i^{p_i}
\cap\HH_i;\\  
\pinf,&\text{otherwise},
\end{cases}
\end{align}
and $p_i\in\left]1,\pinf\right[$.
This modification is motivated by 
bonding problems in continuum mechanics. 
\item Note that $(\overline{g}_{ij})_{(i,j)\in K}$ defined in 
\eqref{e:gijplap} is a solution to the dual problem associated with 
Problem~\ref{prob:plap}. The method proposed in 
Theorem~\ref{t:plap} also converges in the dual variables, but 
for the sake of simplicity we provide only the convergence in primal 
variables.
\end{enumerate}
\end{remark}

\begin{remark}\label{Plat}
The Plateau problem, i.e., the non parametric zero mean 
curvature problem, can be treated similar to the $p$-Laplacian 
problem (case $1<p<2$). The variational problem reads
\begin{equation}
\label{Plat1}
\minimize{\substack{u\in W^{1,1}(\Omega)\\u=\phi\:\:\text{on}\:
\bdry\Omega}} {\int_{\Omega}\sqrt{1+|Du|^2}dx},
\end{equation}
where $\phi\colon\bdry\Omega\to\RR$ is a given boundary data. 
The main issue in that situation is the existence and regularity 
of the solution of the variational problem. The regularity of the 
solution to \eqref{Plat1} has been the object of active research. 
When $\bdry\Omega$ is regular with nonnegative mean 
curvature and $\phi \in C^3(\bar{\Omega})$, there exists a unique 
solution of problem \eqref{Plat1} which is regular, and the boundary 
condition is satisfied in a classical sense (by contrast with the 
relaxed boundary condition in the general case), see 
\cite[Theorem~2.2, pp.~130]{Ekel74}. 
Then one has to modify the function 
$\varphi_i$ by introducing the non homogeneous Dirichlet boundary 
condition in its domain (i.e., $\varphi_i$ is set to $\pinf$ when 
this condition is not satisfied). The function $\varphi_i$ is still 
convex and lower semicontinuous on $\HH_i=H^1(\Omega_i)$.
\end{remark}

\subsection{Obstacle problem}
\label{obst}
We adopt the notation of the Poisson Problem~\ref{prob:Poisson}.
Let $h\colon\Omega\to\RR$ be an obstacle function  
of class $C^{1,1}$, and suppose that the constraint set
\begin{equation}
\label{obst1}
C=\menge{u\in H_0^1(\Omega)}{u\geq h\:\:\text{a.e. in }\Omega}
\end{equation}
is nonempty. 
This clearly requires that $h\leq 0$ on $\bdry\Omega$.

We consider the convex minimization problem called obstacle problem 
\begin{equation}
\label{obst2}
\minimize{u\in C}{\Frac12\int_{\Omega}|Du|^2-
\int_{\Omega}fu}.
\end{equation}
This strongly convex minimization problem admits a unique solution 
$u$ (see \cite{AP,KL} for a general presentation and analysis of
this problem). We are interested in solving it using the following 
equivalent formulation, which fits in our domain decomposition 
approach.

\begin{problem}
\label{prob:obst3}
Consider the setting of Problem~\ref{prob:1}. Let 
$f\in L^2(\Omega)$,
let $h\in C^{1,1}(\Omega)$, and, for every
$i\in I$, define $C_i=\menge{u\in\HH_i}
{u\geq h\:\:\text{a.e. in }\Omega_i}$. Suppose that, 
for every $(i,j)\in K$, $\Upsilon_{ij}$ and $\bdry\Omega$
are of class ${\EuScript C}^2$. The problem is to 
\begin{equation}
\label{obst4}
\minimize{\substack{(u_i)_{i\in I}\in\bigtimes_{i\in I} C_i\\
(\forall (i,j)\in K)\:
\mathsf{T}_{ij}u_i=\mathsf{T}_{ji}u_j}}{ 
\sum_{i=1}^m\Frac12\int_{\Omega_i}|Du_i|^2-\int_{\Omega_i}fu_i}.
\end{equation}
\end{problem}

\begin{proposition}
\label{p:1obst}
Problem~\ref{prob:obst3} has a unique solution 
$(\overline{u}_i)_{i\in I}$. Moreover, the function defined in
\eqref{e:defugen}
is the unique solution to \eqref{obst2}. 
\end{proposition}
\begin{proof}
This is a consequence of Proposition~\ref{p:gen} with, 
for every $i\in I$, 
$\phi_i\colon u\mapsto\iota_{C_i}(u)+\frac12\int_{\Omega_i}|Du|^2-
\int_{\Omega_i}fu$, which are strongly convex. 
In this case, 
$\phi\colon u\mapsto\iota_C(u)+\frac12\int_{\Omega}|Du|^2-
\int_{\Omega}fu$.
\end{proof}

\begin{theorem}
\label{t:obst}
In algorithm \eqref{e:erice2013-06-16a} of Theorem~\ref{t:1}, 
replace the steps defining $p_{i,n}$ and $q_{ij,n}$ by
\begin{equation}
\label{e:proxobst}
p_{i,n}=P_{C_i}\bigg(\Frac{1}{1+\gamma_n}v_{i,n}+
\Frac{\gamma_n}{1+\gamma_n}Q_i(f,0,\ldots,0)\bigg)
\quad\text{and}\quad
q_{ij,n}=0,
\end{equation}
respectively. Then, for every $i\in I$, the 
sequence $(u_{i,n})_{n\in\NN}$ generated by 
\eqref{e:erice2013-06-16a} 
converges strongly to $\overline{u}_i$ in $\HH_i$.
\end{theorem}
\begin{proof}
Set
\begin{equation}
\label{e:deffctobst}
\begin{cases}
(\forall i\in I)\quad
\displaystyle{\varphi_i\colon\HH_i\to\RX\colon  u_i 
\mapsto\iota_{C_i}(u_i)+\Frac12\int_{\Omega_i}|Du_i|^2-
\int_{\Omega_i}fu_i}\\
(\forall (i,j)\in K)\quad \psi_{ij}=\iota_{\{0\}}.
\end{cases}
\end{equation}
Since the sets $(C_i)_{i\in I}$ are 
closed and convex in $\HH_i$, the convex functions
$(\varphi_i)_{i\in I}$ 
are lower semicontinuous, and hence, for every $i\in I$, 
$\varphi_i\in\Gamma_0(\HH_i)$.
Moreover, for every $(i,j)\in K$, 
$\psi_{ij}\in\Gamma_0 (L^2(\Upsilon_{ij}))$. Altogether, 
Problem~\ref{prob:obst3} is a particular case of 
Problem~\ref{prob:1}. Let us verify that condition 
\eqref{e:exicond} holds.
Let $(\overline{u}_i)_{i\in I}\in C_1\times\cdots\times C_m$ 
be the solution to Problem~\ref{prob:obst3}, and let 
$\overline{u}\in C$ 
defined by \eqref{e:defugen} be the unique solution to
\eqref{obst2}
guaranteed by Proposition~\ref{p:1obst}. Since $\psi_{ij}=
\iota_{\{0\}}$, we have $\partial\psi_{ij}(0)=L^2(\Upsilon_{ij})$, 
and hence the first condition in \eqref{e:exicond} is satisfied. 
Since $\bdry\Omega$ and $(\Upsilon_{ij})_{(i,j)\in K}$ are
of class ${\EuScript C}^2$ and $h\in C^{1,1}$, we have
$\overline{u}\in C^{1,1}$ and, for every $i\in I$ and $j\in J(i)$, 
$\nu_i^\top D\overline{u}_i\in L^2(\Upsilon_{ij})$ 
and $\nu_j^\top D\overline{u}_j\in L^2(\Upsilon_{ij})$
\cite[Theorem~8.2]{KL} (see also \cite{Fr}).
Now let us show that the second condition in 
\eqref{e:exicond} holds with 
\begin{equation}
\label{e:gijobst}
(\forall (i,j)\in K)\quad
\overline{g}_{ij}=\nu_j^\top D\overline{u}_j\in L^2(\Upsilon_{ij}).
\end{equation}
The optimality condition for the  solution $\overline{u}$
to \eqref{obst2} and 
Proposition~\ref{p:1obst} yield $\overline{u}\in C$ and
\begin{equation}
(\forall v\in C)\quad 
\int_{\Omega}D\overline{u}^{\top}D(v-\overline{u})
-\int_{\Omega}f(v-\overline{u})\geq0
\end{equation}
or, equivalently,
\begin{multline}
\label{e:optobst}
(\forall i\in I)(\forall v_i\in C_i)\:\:\text{such that}\:\:
(\forall (i,j)\in K)
\:\mathsf{T}_{ij}v_i=\mathsf{T}_{ji}v_j\\
\sum_{i\in I}\bigg(\int_{\Omega_i}D\overline{u}_i^\top 
D(v_i-\overline{u}_i)-\int_{\Omega_i}f(v_i-\overline{u}_i)\bigg)\geq
0.
\end{multline}
Using \eqref{e:Q_i} and integration by parts, \eqref{e:optobst} 
is written as 
\begin{multline}
\label{e:optobst2}
(\forall i\in I)(\forall v_i\in C_i)\text{  such that  }
(\forall (i,j)\in K)
\:\mathsf{T}_{ij}v_i=\mathsf{T}_{ji}v_j\\
\sum_{i\in I}\bigg(\int_{\Omega_i}D\overline{u}_i^\top 
D(v_i-\overline{u}_i)+\int_{\bdry\Omega_i}\nu_i^{\top}
Dw_i\cdot\mathsf{T}_i(v_i-\overline{u}_i)dS-\int_{\Omega_i}
Dw_i^{\top}D(v_i-\overline{u}_i)\bigg)\geq 0,
\end{multline}
where $w_i=Q_i(f,(-\overline{g}_{ij})_{j\in
J(i+)},(-\overline{g}_{ji})_{j\in J(i-)})$ or, equivalently,
\begin{multline}
\label{e:optobst3}
(\forall i\in I)(\forall v_i\in C_i)\text{  such that  }
(\forall (i,j)\in K)
\:\mathsf{T}_{ij}v_i=\mathsf{T}_{ji}v_j\\
\sum_{i\in I}\bigg(\int_{\Omega_i}D(\overline{u}_i-w_i)^\top 
D(v_i-\overline{u}_i)-\sum_{j\in
J(i+)}\int_{\Upsilon_{ij}}\overline{g}_{ij}
\cdot\mathsf{T}_{ij}(v_i-\overline{u}_i)dS\\+\sum_{j\in
J(i-)}\int_{\Upsilon_{ij}}\overline{g}_{ji}
\cdot\mathsf{T}_{ij}(v_i-\overline{u}_i)dS\bigg)\geq 0.
\end{multline}
Since, for every $(i,j)\in K$,
$\mathsf{T}_{ij}v_i=\mathsf{T}_{ji}v_j$ and 
$\mathsf{T}_{ij}\overline{u}_i=\mathsf{T}_{ji}\overline{u}_j$ we have
\begin{align}
\sum_{i\in I}\sum_{j\in J(i-)}\int_{\Upsilon_{ij}}\overline{g}_{ji}
\cdot\mathsf{T}_{ij}(v_i-\overline{u}_i)dS
&=\sum_{j\in I}\sum_{i\in J(j+)}\int_{\Upsilon_{ij}}\overline{g}_{ji}
\cdot\mathsf{T}_{ij}(v_i-\overline{u}_i)dS\nonumber\\
&=\sum_{j\in I}\sum_{i\in J(j+)}\int_{\Upsilon_{ji}}\overline{g}_{ji}
\cdot\mathsf{T}_{ji}(v_j-\overline{u}_j)dS,
\end{align}
and, hence, \eqref{e:optobst3} reduces to 
\begin{multline}
\label{e:optobst4}
(\forall i\in I)(\forall v_i\in C_i)\text{  such that  }(\forall 
(i,j)\in K)\:\mathsf{T}_{ij}v_i=\mathsf{T}_{ji}v_j\\
\sum_{i=1}^m\int_{\Omega_i}D(\overline{u}_i-w_i)^\top 
D(v_i-\overline{u}_i)\geq 0.
\end{multline}
Now fix $i\in I$. Taking appropriate functions $(v_j)_{j\neq i}$, we 
deduce from \eqref{e:optobst4} that
\begin{align}
(\forall v_i\in C_i)\quad
\scal{\overline{u}_i-w_i}{v_i-\overline{u}_i}=\int_{\Omega_i}
D(\overline{u}_i-w_i)^\top 
D(v_i-\overline{u}_i)\geq 0,
\end{align}
which is equivalent to $w_i-\overline{u}_i
=Q_i(f,(-\overline{g}_{ij})_{j\in J(i+)},(-\overline{g}_{ji})_{j\in
J(i-)})
-\overline{u}_i\in N_{C_i}(\overline{u}_i)$.
Hence, using the linearity of the operator $Q_i$, we obtain
\begin{equation}
Q_i(f,0,\ldots,0)-
Q_i(0,(\overline{g}_{ij})_{j\in J(i+)},(\overline{g}_{ji})_{j\in
J(i-)})
\in \overline{u}_i+N_{C_i}(\overline{u}_i).
\end{equation}
We deduce from Proposition~\ref{p:13}\ref{p:13i} that
\begin{equation}
-Q_i(0,(\overline{g}_{ij})_{j\in J(i+)},(\overline{g}_{ji})_{j\in
J(i-)})\in
\partial\varphi_i(\overline{u}_i). 
\end{equation}
Hence, \eqref{e:exicond} holds.
On the other hand, it follows from \eqref{e:prox} and 
\eqref{e:deffctobst} that, for every $(i,j)\in K$ and $\mu\in\RPP$, 
$\prox_{\mu\psi_{ij}}\equiv0$. Hence, we deduce from 
Proposition~\ref{p:13}\ref{p:13i} that \eqref{e:proxobst}
yields
\begin{equation}
\label{e:proxobst2}
(\forall n\in\NN)\quad 
\begin{cases}
(\forall i\in I)\quad 
p_{i,n}=\prox_{\gamma_n\varphi_i}v_{i,n}\\
(\forall (i,j)\in K)\quad 
q_{ij,n}=\prox_{\mu_n\psi_{ij}}(l_{ij,n}+\mu_ng_{ij,n}).
\end{cases}
\end{equation}
Therefore, the result follows from 
Theorem~\ref{t:1}, where $(\varphi_i)_{i\in I}$ and
$(\psi_{ij})_{(i,j)\in K}$ are defined by 
\eqref{e:deffctobst}.
\end{proof}

\newpage
\begin{remark}\
\begin{enumerate}
\item
Note that $(\overline{g}_{ij})_{(i,j)\in K}$ defined in 
\eqref{e:gijobst} is a solution to the dual problem associated with 
Problem~\ref{prob:obst3}. The method proposed in 
Theorem~\ref{t:obst} also guarantees the convergence of the 
dual variables, but for the sake of simplicity we provide only 
the primal convergence statement.
\item In \eqref{e:erice2013-06-16a} we have
\begin{equation}
(\forall n\in\NN)(\forall i\in I)\quad v_{i,n}=u_{i,n}-\gamma_nQ_i
\big(0,(g_{ij,n})_{j\in J(i+)},(g_{ji,n})_{j\in J(i-)}\big).
\end{equation}
Hence, since the operators $(Q_i)_{i\in I}$ defined in \eqref{e:Q_i} 
are multilinear, the sequences $(p_{i,n})_{i\in I,\, n\in\NN}$ 
can be computed more efficiently via
\begin{equation}
(\forall n\in\NN)
(\forall i\in I)\quad 
p_{i,n}=P_{C_i}\Big(\Frac{1}{1+\gamma_n}u_{i,n}+
\Frac{\gamma_n}{1+\gamma_n}Q_i(f,(-g_{ij,n})_{j\in 
J(i+)},(-g_{ji,n})_{j\in J(i-)})\Big).
\end{equation}
This allows us to solve only $m$ auxiliary PDE's 
for updating $(p_{i,n})_{i\in I}$ at each iteration $n$.
\end{enumerate}
\end{remark}

\section{Perspectives}
\label{sec:5}
In this section we briefly outline possible adaptations and 
variants of our framework to related problems. 

First, in the setting of the Poisson 
Problem~\ref{prob:Poisson} let, for 
every $i\in I$ and $j\in J(i+)$, $\varepsilon_{ij}\in\{-1,1\}$,
and consider the variational problem
\begin{equation}
\label{uni1}
\minimize{\substack{u_1\in\HH_1,\ldots,u_m\in\HH_m  \\
\vspace{3mm}
(\forall i\in I)(\forall j\in J(i+))\:\,
\varepsilon_{ij}\left(\mathsf{T}_{ij}u_i-\mathsf{T}_{ji}u_j\right)
\geq 0}}
{\sum_{i=1}^m\Frac12\int_{\Omega_i}|Du_i|^2-\int_{\Omega_i}fu_i}.
\end{equation}
By contrast with the preceding problems in which the bilateral 
constraint $\mathsf{T}_{ij}u_i-\mathsf{T}_{ji}u_j=0$ imposes a 
continuity property at the interfaces, the constraint
$\varepsilon_{ij}\left(\mathsf{T}_{ij}u_i-\mathsf{T}_{ji}u_j\right)
\geq 0$ models a unilateral transmission condition through the 
interfaces. This occurs for example in the modelling of fissures
and cracks. Depending on the sign of 
$\varepsilon_{ij}$, we have a nonzero flux from $\Omega_i$ towards 
$\Omega_j$, or in the reverse direction. The main difference with 
respect to the previous examples is that, instead of using 
$\psi_{ij}=\iota_{\{0\}}$, in this case we set
$\psi_{ij}=\iota_{\{L^2(\Upsilon_{ij})^+\}}$ 
or $\psi_{ij}=\iota_{\{L^2(\Upsilon_{ij})^-\}}$,
depending on the sign of $\varepsilon_{ij}$. Clearly
$\psi_{ij}\in\Gamma_0(L^2(\Upsilon_{ij}))$ because 
$L^2(\Upsilon_{ij})^+$ and $L^2(\Upsilon_{ij})^-$ are closed convex 
cones in $L^2(\Upsilon_{ij})$.

Modeling semi-permeable membranes gives rise to similar problems, 
which possibly involve both unilateral transmission conditions and 
surface energy functions. For example (here $\mu_{ij}>0$ stands
for some permeability coefficients)
\begin{equation}
\label{uni2}
\minimize{\substack{i\in I,\:u_i\in\HH_i \\
\varepsilon_{ij}\left(\mathsf{T}_{ij}u_i-\mathsf{T}_{ji}u_j\right)
\geq 0,\\j\in J(i+) }}
{\sum_{i=1}^m\Frac12\int_{\Omega_i}|Du_i|^2-\int_{\Omega_i}fu_i} 
+{\sum_{i,j}\Frac{\mu_{ij}}{2}\int_{\Upsilon_{ij}}|\mathsf{T}_{ij}
u_i-\mathsf{T}_{ji}u_j|^2}.
\end{equation}
This problem is within the scope of our study. Depending on the sign 
of $\varepsilon_{ij}$ one can take 
\begin{equation}
\psi_{ij}(g)=\iota_{\{L^2(\Upsilon_{ij})^+\}}(g)+
\Frac{\mu_{ij}}{2}\int_{\Upsilon_{ij}}|g|^2
\end{equation}
or
\begin{equation}
\psi_{ij}(g)=\iota_{\{L^2(\Upsilon_{ij})^-\}}(g)+ 
\Frac{\mu_{ij}}{2}\int_{\Upsilon_{ij}}|g|^2.
\end{equation}

Finally, let us note that in this paper we have considered 
only Dirichlet boundary conditions. Neumann and mixed boundary 
conditions can also be considered by working in Sobolev spaces 
$(\HH_i)_{i\in I}$ associated with the corresponding variational 
formulation (for example, for the Neumann problem, one can take 
$\HH_i=H^1(\Omega_i)$).

\end{document}